\documentclass[a4paper,10pt]{amsart}

\usepackage{verbatim, amssymb, enumerate}

\renewcommand \a{\alpha}
\renewcommand \b{\beta}
\newcommand \K{\delta}
\newcommand \n{\nabla}
\newcommand \la{\lambda}
\newcommand \ve{\varepsilon}
\newcommand \id{\mathrm{id}}
\newcommand \br{\mathbb{R}}
\newcommand \bc{\mathbb{C}}
\newcommand \Oc{\mathbb{O}}

\newcommand \rk{\operatorname{rk}}
\newcommand \Ker{\operatorname{Ker}}

\newcommand \End{\operatorname{End}}

\newcommand \Span{\operatorname{Span}}
\newcommand \Tr{\operatorname{Tr}}
\newcommand \db{\partial}

\newcommand \cp{\mathcal{C}}

\newcommand \so{\mathfrak{so}}
\newcommand\ag{\mathfrak a}

\newcommand\g{\mathfrak g}
\newcommand\h{\mathfrak h}

\newcommand\m{\mathfrak m}
\renewcommand\i{\mathrm i}

\newcommand \Sym{\operatorname{Sym}}
\newcommand \Hom{\operatorname{Hom}}

\newcommand \ad{\operatorname{ad}}

\newcommand \<{\langle}
\renewcommand \>{\rangle}

\newcommand \mU{\mathcal{U}}

\makeatletter
\newtheorem*{rep@theorem}{\rep@title}
\newcommand{\newreptheorem}[2]{%
\newenvironment{rep#1}[1]{%
\def\rep@title{#2 \ref{##1}}%
\begin{rep@theorem}}%
{\end{rep@theorem}}}
\makeatother

\theoremstyle{plane}

\newtheorem*{theorem*}{Theorem}
\newtheorem*{corollary*}{Corollary}
\newtheorem{lemma}{Lemma}
\newtheorem{proposition}{Proposition}
\newtheorem*{conj*}{Conjecture}
\newtheorem*{prop*}{Proposition}
\newreptheorem{lemma}{Lemma}

\theoremstyle{definition}

\newtheorem*{definition*}{Definition}

\theoremstyle{remark}
\newtheorem{remark}{Remark}

\begin{document}

\title{Weyl-Schouten Theorem for symmetric spaces}

\author{Y.Nikolayevsky}
\address{Department of Mathematics and Statistics, La Trobe University, Victoria, 3086, Australia}
\email{y.nikolayevsky@latrobe.edu.au}

\date{\today}



\begin{abstract}
Let $M_0$ be a symmetric space of dimension $n  > 5$ whose de Rham decomposition contains no factors of constant curvature and let $W_0$ be the Weyl tensor of $M_0$ at some point. We prove that a Riemannian manifold $M$ whose Weyl tensor at every point is a positive multiple of $W_0$ is conformally equivalent to $M_0$ (the case $M_0 = \br^n$ is the Weyl-Schouten Theorem).
\end{abstract}

\maketitle

\section{Introduction}
\label{s:intro}

In this paper we generalise the classical Weyl-Schouten Theorem to the case when the model space is a Riemannian symmetric space and also consider the notion of curvature homogeneity in terms of the Weyl conformal curvature tensor.

A smooth Riemannian manifold $M^n$ is called \emph{curvature homogeneous}, if for any two points $x, y \in M^n$, there exists a linear isometry $\iota:T_xM^n \to T_yM^n$ which maps the curvature tensor of $M^n$ at $x$ to the curvature tensor of $M^n$ at $y$. A smooth Riemannian manifold $M^n$ is \emph{modelled on a homogeneous space $M_0$}, if for every point $x \in M^n$, there exists a linear isometry $\iota:T_xM^n \to T_oM_0$ which maps the curvature tensor of $M^n$ at $x$ to the curvature tensor of $M_0$ at $o \in M_0$ (the manifold $M^n$ is then automatically curvature homogeneous). The term ``curvature homogeneous" was introduced by F.Tricerri and L.Vanhecke in 1986 \cite{TV}; for the current state of knowledge the reader is referred to \cite{Gil}.

\begin{definition*} \label{d:ws}
A smooth Riemannian manifold $M^n$ is called \emph{Weyl homogeneous}, if for any $x, y \in M^n$, there exists a linear isometry
$\iota:T_xM^n \to T_yM^n$ which maps the Weyl tensor of $M^n$ at $x$ to a positive multiple of the Weyl tensor of $M^n$ at $y$.
A smooth Weyl homogeneous Riemannian manifold $M^n$ is \emph{modelled on a homogeneous space $M_0$}, if for every point
$x \in M^n$, there exists a linear isometry $\iota:T_xM^n \to T_oM_0$ which maps the Weyl tensor of $M^n$ at $x$ to a
positive multiple of the Weyl tensor of $M_0$ at $o \in M_0$.
\end{definition*}

In the latter case, we say that $M^n$ has \emph{the same} Weyl tensor as $M_0$. A Riemannian manifold which is conformally equivalent to a (model) homogeneous space is trivially Weyl homogeneous. One may ask if the converse is true, namely,
\begin{quote}
\emph{is a Riemannian manifold having the same Weyl conformal curvature tensor as a homogeneous space $M_0$ locally conformally equivalent to $M_0$?}
\end{quote}
By the classical Weyl-Schouten Theorem, the answer is in positive, when $M_0=\br^n, \, n \ge 4$ (of course, the restriction $\dim M_0 >3$ is implicitly assumed in the question). In general, however, the answer is in negative even for Weyl homogeneous manifolds modelled on symmetric spaces (see \cite[Section~4]{Ndga}, where an example from \cite[Theorem~4.2]{BKV} is discussed from the conformal point of view). Moreover, based on the existence of many examples of curvature homogeneous manifolds which are not locally homogeneous \cite{Gil} one would most probably expect at least as many examples in the conformal settings.

Nevertheless, the situation is not that hopeless, as a curvature homogeneous manifold modelled on a symmetric space \emph{is}, in the most cases, locally isometric to its model space. More precisely, by \cite[Corollary~10.3]{KTV}, this is true, \emph{provided the de Rham decomposition of the model space contains no product of the form $M^2(\kappa) \times \br^m$}, where $M^2(\kappa)$ is a Riemannian manifold of dimension two of constant curvature $\kappa \ne 0$ and $m \ge 1$.

Our main result states that the picture in the conformal settings is somewhat similar:
\begin{theorem*}
Let $M_0$ be a Riemannian symmetric space of dimension $n > 5$ whose de Rham decomposition contains no factors of constant curvature. Then any smooth Weyl homogeneous Riemannian manifold modelled on $M_0$ is locally conformally equivalent to $M_0$.
\end{theorem*}

In the earlier papers, the Theorem was established for $M_0=\bc P^m, \; m \ge 4$ (and for its noncompact dual) \cite{BG1}, for rank-one symmetric spaces of dimension $n > 4$ \cite[Theorem~2]{Nampa}, and for simple groups with a bi-invariant metric \cite{Ndga}. The dimension restriction in the Theorem excludes only the spaces $\bc P^2$ and $SU(3)/SO(3)$ and their duals. Note that the claim of the Theorem is false in the case $M_0=\bc P^2$, as a four-dimensional Riemannian manifold having the same Weyl tensor as $\bc P^2$ is either self-dual or anti-self-dual by \cite{BG2} and as there exist self-dual K\"{a}hler metrics on $\bc^2$ which are not locally conformally equivalent to any locally symmetric one \cite{Der}. In the case $M_0=SU(3)/SO(3)$, we show that at least the infinitesimal version of the Theorem (Proposition~\ref{p:main} in Section~\ref{s:geo}) is not satisfied (see Section~\ref{s:su3so3}).

\smallskip

The paper is organised as follows. In Section~\ref{s:geo} we give a brief introduction following the setup of \cite{Ndga} and then prove the Theorem with the help of Lie-algebraic Proposition~\ref{p:main}. The rest of the paper (except for Section~\ref{s:su3so3}) is devoted to the proof of Proposition~\ref{p:main}. In Section~\ref{s:proof}, we reduce the proof of Proposition~\ref{p:main} to the case when $M_0$ is irreducible (Proposition~\ref{p:irre}). Further on, in Section~\ref{s:irred}, we prove Proposition~\ref{p:irre} using three technical lemmas: Lemma~\ref{l:Amh}, Lemma~\ref{l:rkge3} (which covers the case $\rk M_0 \ge 3$) and Lemma~\ref{l:rk2} (the case $\rk M_0 = 2$). The former two are proved in Section~\ref{s:rr}, the latter one, in Section~\ref{s:rk2}. The proof is completed in Section~\ref{s:rk1}, where we consider the rank one spaces. In Section~\ref{s:su3so3}, the final one, we show that Proposition~\ref{p:main} is false for $M_0=SU(3)/SO(3)$.

\section{Proof of the Theorem}
\label{s:geo}

Let $M^n$ be a Riemannian manifold with the metric $\<\cdot, \cdot\>$ and the Levi-Civita connection $\n$. For vector fields $X,Y$ define the field of a linear operator $X \wedge Y$ by lowering the index of the corresponding bivector: $(X \wedge Y)Z=\<X,Z\>Y-\<Y,Z\>X$. The curvature tensor is defined by $R(X,Y)=\n_X\n_Y-\n_Y\n_X-\n_{[X,Y]}$, where $[X,Y]=\n_XY-\n_YX$, and the Weyl conformal curvature tensor $W$, by
\begin{equation}\label{eq:weyldef}
R(X,Y)=(\rho X) \wedge Y + X \wedge (\rho Y) + W(X, Y),
\end{equation}
where $\rho =\frac{1}{n-2}\operatorname{Ric}-\frac{\operatorname{scal}}{2(n-1)(n-2)}\id$ is the Schouten tensor, $\operatorname{Ric}$ is the Ricci operator and $\operatorname{scal}$ is the scalar curvature. Denote $W(X,Y,Z,V)=\<W(X,Y)Z,V\>$. We have the following easy lemma.

\begin{lemma}[{\cite[Lemma~1]{Ndga}}]\label{l:eismooth}
Suppose that $M^n$ is a Weyl homogeneous manifold with the metric $\<\cdot, \cdot\>'$ modelled on a homogeneous space $M_0$ with the Weyl tensor $W_0 \ne 0$. Choose a point $o \in M_0$ and an orthonormal basis $E_i$ for $T_oM_0$. Then there exists a smooth metric $\<\cdot, \cdot\>$ on $M^n$ conformally equivalent to $\<\cdot, \cdot\>'$ such that for every $x \in M^n$, there exists a smooth orthonormal frame $e_i$ \emph{(}relative to $\<\cdot, \cdot\>$\emph{)} on a neighbourhood $\mU=\mU(x) \subset M^n$ satisfying $W(e_i,e_j,e_k,e_l)(y)=W_0(E_i,E_j,E_k,E_l)$, for all $y \in \mU$.
\end{lemma}

\begin{remark}\label{rem:wne0}
In our case, the condition $W_0 \ne 0$ is satisfied, as otherwise $M_0$ were locally isometric to one of the spaces $\br^n, \, \br \times S^{n-1}(\kappa), \, \br \times H^{n-1}(-\kappa), \, S^{n-p}(\kappa) \times H^{p}(-\kappa), \; \kappa >0, \; 0 \le p \le n$ \cite{Kur}, which would contradict the assumption that $M_0$ has no factors of constant curvature.
\end{remark}

For the remainder of the proof, we assume that the metric on $M^n$ is chosen as in Lemma~\ref{l:eismooth} and we will be proving that
$M^n$, with that metric, is locally isometric to the model space $M_0$.

Let $x \in M^n$ and let $e_i$ be the orthonormal frame on the neighbourhood $\mU$ of $x$ introduced in Lemma~\ref{l:eismooth}.
For every $Z \in T_xM^n$, define the linear operator $K_Z$ (the connection operator) by
\begin{equation}\label{eq:gdefK}
K_Ze_i=\n_Ze_i,
\end{equation}
and extended to $T_xM^n$ by linearity. As the basis $e_i$ is orthonormal, $K_Z$ is skew-symmetric.
For smooth vector fields $X, Y$ on $\mU$, define the Cotton-York tensor (up to a constant multiple), by
\begin{equation}\label{eq:gdefPhi}
\Phi(X,Y)=(\n_X \rho)Y-(\n_Y \rho)X,
\end{equation}
where $\rho$ is the Schouten tensor (see \eqref{eq:weyldef}). Clearly $\Phi$ is skew-symmetric and
\begin{equation}\label{eq:cyclePhi}
\sigma_{XYZ}\<\Phi(X,Y),Z\>=0,
\end{equation}
where $\sigma_{XYZ}$ is the sum over the cyclic permutations of the triple $(X,Y,Z)$.

Let $M_0=G/H$ be the model symmetric space for $M^n$, where $G$ is the identity component of the full isometry group of $M_0$ and $H$ is the isotropy subgroup of $o \in M_0$, and let $\g=\h+\m$ be the corresponding Cartan decomposition, where $\g$ and $\h$ are the Lie algebras of $G$ and $H$ respectively, and $\m=T_oM_0$. Denote $R_0$ the curvature tensor of $M_0$ at $o$, so that for $X, Y, Z \in T_oM_0, \; R_0(X,Y)Z=-[[X,Y],Z]= -\ad_{[X,Y]}Z$. We denote $\ad(\h) \subset \so(\m)$ the isotropy subalgebra of $M_0$ at $o$.

In the assumptions of Lemma~\ref{l:eismooth}, identify $T_xM^n$ with $T_oM_0$ via the linear isometry $\iota$ mapping $e_i$ to $E_i$. Define $K$ and $\Phi$ on $\m=T_oM_0$ by the pull-back by $\iota$.

Let $\operatorname{Ric_0}$ and $\operatorname{scal}_0$ be the Ricci tensor and the scalar curvature of $M_0$ (at $o \in M_0$), and let $\rho_0=\frac{1}{n-2}\operatorname{Ric_0}-\frac{\operatorname{scal}_0}{2(n-1)(n-2)}\id$ (see \eqref{eq:weyldef}). Define the operator $\Psi: \Lambda^2\m \to \m$ by
\begin{equation}\label{eq:defPsi}
\Psi(X,Y)=\Phi(X,Y) + [\rho_0,K_X]Y - [\rho_0,K_Y]X,
\end{equation}
where (here and below) the bracket of linear operators is the usual commutator. From \eqref{eq:cyclePhi} and the fact that $[\rho_0,K_X]$ is symmetric it follows that
\begin{equation}\label{eq:cyclePsi}
\sigma_{XYZ}\<\Psi(X,Y),Z\>=0, \quad \text{for all} \quad  X,Y,Z \in \m.
\end{equation}

\begin{lemma}[{\cite[Lemma~2]{Ndga}}]\label{l:gbianchi}
In the assumptions of Lemma~\ref{l:eismooth}, let $M_0$ be a symmetric space. For $x \in M^n$, identify $T_xM^n$ with $\m=T_oM_0$ via the linear isometry $\iota$ mapping $e_i$ to $E_i$. Define $K$ and $\Phi$ on $T_xM^n$ by (\ref{eq:gdefK}, \ref{eq:gdefPhi}) and on $T_oM_0$, by the pull-back by $\iota$, and define $\Psi$ by \eqref{eq:defPsi}. Then
\begin{gather}\label{eq:gbiad}
\sigma_{XYZ}([\ad_{[X, Y]},K_Z]+\ad_{[K_XY-K_YX, Z]}+ \Psi(X,Y)\wedge Z)=0, \\
(\n_Z W)(X,Y)=[\ad_{[X, Y]},K_Z]+\ad_{[K_ZX, Y]-[K_ZY, X]}+([\rho_0,K_Z]X) \wedge Y + X \wedge ([\rho_0,K_Z]Y). \label{eq:gnablaW}
\end{gather}
\end{lemma}

Equation \eqref{eq:gbiad} is just the second Bianchi identity. Note that the expression for $(\n_Z W)(X,Y)$ given in \cite[Eq.~(8)]{Ndga} has an unfortunate typo (which affects neither the result, nor the proof); the correct form of the right-hand side is the one given in \eqref{eq:gnablaW}.

We deduce the Theorem from the following proposition whose proof is given in Sections~\ref{s:proof} and \ref{s:irred}.

\begin{proposition} \label{p:main}
Let $M_0=G/H$ be a Riemannian symmetric space of dimension $n > 5$ with no factors of constant curvature and let $\g=\h\oplus\m, \; \m=T_oM_0$, be the corresponding Cartan decomposition. Suppose that the maps $K \in \Hom(\m, \so(\m)), \; K: Z \mapsto K_Z$, and $\Psi \in \Hom(\Lambda^2\m, \m), \; \Psi: X \wedge Y \mapsto \Psi(X,Y)$, satisfy \eqref{eq:gbiad} and \eqref{eq:cyclePsi}, for all $X,Y,Z \in \m$. Then $\Psi=0$ and  $K \in \Hom(\m, \ad(\h))$. \end{proposition}

\begin{proof}[Proof of the Theorem assuming Proposition~\ref{p:main}]
Let $\m=\oplus_{s=1}^N \m_s, \; N \ge 1$, be the orthogonal decomposition corresponding to the de Rham decomposition of $M_0$. From Proposition~\ref{p:main}, $K_Z \in \ad(\h)$, for all $Z \in \m$, so the sum of the first two terms on the right-hand side of \eqref{eq:gnablaW} vanishes. The same is true for the last two terms, as every $\m_s$ is an invariant subspace of $K_Z \in \ad(\h)$ and as the restriction of $\rho_0$ to each $\m_s$ is a multiple of the identity (as the irreducible factors are Einstein). It then follows from \eqref{eq:gnablaW} that $\n W=0$. By \cite{Rot}, as $\n W=0$, but $W \ne 0$ (see Remark~\ref{rem:wne0}), the manifold $M^n$ is locally symmetric. To prove that $M^n$ is locally isometric to $M_0$, it suffices to show that $R=R_0$. As the Weyl tensors of $M^n$ and $M_0$ are equal, it suffices to show that $\rho=\rho_0$, by \eqref{eq:weyldef}.

For a symmetric operator $A \in \Sym(\m)$ define $S_A \in \Sym(\so(\m))$ by $S_A T=A T + T A$, for $T \in \so(\m)$. Viewing $R, R_0$ and $W$ as the elements of $\Sym(\so(\m))$ we have $R=W+S_\rho$ and $R_0=W+S_{\rho_0}$ by \eqref{eq:weyldef}. As the space $M$ is symmetric, hence is locally a product of Einstein spaces, we have $[R,S_\rho]=0$, so $[W,S_\rho]=0$. Moreover, as $\n W=0$, we have $R(T).W=0$, for all $T \in \so(\m)$, where $R(T)$ is viewed as a differentiation of the tensor algebra. It follows that $W(T).W+S_\rho(T).W=0$. Denote $\tau=\rho-\rho_0$. As the above equations also hold for $M_0$, we get by linearity that $[W,S_\tau]=0$ and $S_\tau(T).W=0$. The later equation is equivalent to $W([S_\tau(T),N])+[W(N),S_\tau(T)]=0$, for all $N,T \in \so(\m)$, so we obtain that $\tau$ satisfies the following equations:
\begin{equation}\label{eq:tau}
    [W,S_\tau]=0, \hskip 2cm [W, \ad_{S_\tau(T)}]=0, \quad \text{for all } T \in \so(\m),
\end{equation}
where the brackets and $\ad$ are in sense of the Lie algebra $\mathfrak{gl}(\so(\m))$.

Let $\mathfrak{s}\subset\so(\m)$ be the Lie subalgebra generated by the subspace $S_\tau(\so(\m)) \subset \so(\m)$. Let $e_i, \; i=1, \dots, n$, be orthonormal eigenvectors of $\tau$, and $\la_i$ be the corresponding eigenvalues. Then $e_i \wedge e_j$ is an eigenvector of $S_\tau$, with the eigenvalue $\la_i + \la_j$. It follows that if $\la_i + \la_j \ne 0$, then $e_i \wedge e_j \in S_\tau(\so(\m)) \subset \mathfrak{s}$. Moreover, as $[e_i \wedge e_k, e_k \wedge e_j]=-e_i \wedge e_j$, we obtain that $e_i \wedge e_j \in \mathfrak{s}$, if there exists $k=1, \dots, n$ such that both $\la_i + \la_k$ and $\la_k + \la_j$ are nonzero. It follows that either $\tau =0$, so $\rho=\rho_0$ and we are done; or $\tau$ has exactly two eigenvalues: $\la$, of multiplicity $p$, and $-\la$, of multiplicity $q$, with $\la \ne 0, \; p, q > 0, \; p+q=n$, in which case $\mathfrak{s} =\so(p) \oplus \so(q)$ standardly embedded in $\so(\m)$; or, in all the other cases, $\mathfrak{s} =\so(\m)$. We show that the last two cases imply that $W$ is a multiple of the identity on $\so(\m)$.

Indeed, from the second equation of \eqref{eq:tau}, $W$ commutes with $\ad_T$, for all $T \in \mathfrak{s}$. First, suppose that $\mathfrak{s} =\so(\m)$. As the eigenspaces of an operator on an arbitrary Lie algebra, which commutes with all the $\ad$'s, are ideals and as $\so(\m)$ is simple (since $n > 5$), we get $W=c \, \id_{\so(\m)}$, for some $c \in \br$. Next, suppose that $\mathfrak{s} =\so(p) \oplus \so(q), \; p, q > 0, \; p+q=n$. By relabelling the eigenvectors, we can assume that the $\la$-eigenspace of $\tau$ is spanned by $e_1, \dots, e_p$ and the $(-\la)$-eigenspace, by $e_{p+1}, \dots, e_n$. Then the eigenvalues of $S_\tau$ are $2\la, 0$ and $-2\la$, with the eigenspaces $E_{2\la}=\so(p)=\Span_{s,t \le p} (e_s \wedge e_t), \; E_{-2\la}=\so(q)=\Span_{a,b > p} (e_a \wedge e_b)$ and $E_0=\Span_{s \le p < a} (e_s \wedge e_a)$, respectively. By the first equation of \eqref{eq:tau}, these subspaces are $W$-invariant. First suppose that $p \ne 1, 4$. Then $\so(p)$ is either simple or one-dimensional, so, as the restriction of $W$ to $\so(p)$ commutes with $\ad_{\so(p)}$, we obtain that $W_{|\so(p)}=c \, \id_{\so(p)}$. Then, as $[e_s \wedge e_t, e_t \wedge e_a]=-e_s \wedge e_a$, for $s, t \le p < a$, and as $W$ commutes with $\ad_{\so(p)}$ on the whole $\so(\m)$, we get that  $W_{|E_0}=c \, \id_{E_0}$. If $p=4$, then $\so(p)$ is the direct sum of the ideals $\mathfrak{s}_1, \mathfrak{s}_2$ isomorphic to $\so(3)$ and we have  $W_{|\mathfrak{s}_\alpha}=c_\alpha \, \id_{\mathfrak{s}_\alpha}, \; \alpha=1,2$. As for all nonzero $T \in \mathfrak{s}_1$, the restriction of $\ad_T$ to $E_0$ is nonsingular and as $W$ commutes with $\ad_T$, we obtain that $W_{|E_0}=c_1 \, \id_{E_0}$. Applying the same argument to $\mathfrak{s}_2$ we get $c_1=c_2=c$ and $W_{|E_0}=c \, \id_{E_0}$. Hence for all $p \ne 1$, we have $W_{|E_{2\la} \oplus E_0}=c \, \id_{E_{2\la} \oplus E_0}$. Interchanging $p$ and $q$ and using the fact that $p+q=n \ge 6$ we get $W=c \, \id_{\so(\m)}$.

It follows that $W=c \, \id_{\so(\m)}$, unless $\rho=\rho_0$. But then, as the ``Ricci tensor" of the Weyl tensor vanishes, we have $\sum_i W(X \wedge e_i)e_i =0$, for all $X \in \m$. So $c = 0$, hence $W = 0$, which is a contradiction by Remark~\ref{rem:wne0}.

This proves the Theorem assuming Proposition~\ref{p:main}.
\end{proof}

\section{Proof of Proposition~\ref{p:main}: reducible case}
\label{s:proof}

In this section and in the next section we prove Proposition~\ref{p:main}. We start with the reducible case and prove the following proposition.

\begin{proposition} \label{p:red}
Let $M_0$ be a reducible Riemannian symmetric space with no factors of constant curvature and let $o \in M_0$. Let $\m=T_oM_0 = \oplus_{s=1}^N \m_s, \; N \ge 2$, be the orthogonal decomposition corresponding to the de Rham decomposition of $M_0$. Suppose that the maps $K \in \Hom(\m, \so(\m)), \; K: Z \mapsto K_Z$, and $\Psi \in \Hom(\Lambda^2\m, \m), \; \Psi: X \wedge Y \mapsto \Psi(X,Y)$, satisfy \eqref{eq:cyclePsi} and \eqref{eq:gbiad}, for all $X,Y,Z \in \m$ and that $\Phi$ is defined by \eqref{eq:defPsi}. Then
\begin{enumerate}[\rm 1.]
  \item \label{it:red1}
  $\Psi=\Phi=0$.
  \item \label{it:red2}
  $K \in \Hom(\m, \oplus_{s=1}^N \so(\m_s))$ and for all $s \ne r, \; s,r=1, \dots, N$, there exist linear maps $P_{sr}:\m_r \to \h_s$ such that $(K_Z)_{|\m_s}=\ad_{P_{sr}Z}$, for $Z \in \m_r$.
\end{enumerate}
\end{proposition}

Proposition~\ref{p:red} reduces the proof to the case when $M_0$ is irreducible, which will be treated in Section~\ref{s:irred}. Indeed, if $M_0$ is reducible, the claim of Proposition~\ref{p:main} follows from Proposition~\ref{p:red}, except for the fact that for $Z \in \m_s$, the restriction of $K_Z$ to $\m_s$ belongs to $\ad(\h_s)$. But the projections of $K$ to $\Hom(\m_s, \so(\m_s))$ and of $\Psi$ to $\Hom(\Lambda^2\m_s, \m_s)$ still satisfy \eqref{eq:cyclePsi} and \eqref{eq:gbiad}, for all $X,Y,Z \in \m_s$ (compare to Remark~\ref{r:tg} in Section~\ref{s:irred}), so the fact that $(K_Z)_{|\m_s} \in \ad(\h_s)$ for $Z \in \m_s$ will follow from the proof of Proposition~\ref{p:main} for each irreducible factor separately. Note however, that a reducible $M_0$ may have irreducible factors of dimension five or less, namely the spaces $SU(3)/SO(3)$ and $\bc P^2$ and their duals. For these spaces as such, the claim of Proposition~\ref{p:main} is false; this follows from the dimension count for $\bc P^2$ and from the results of Section~\ref{s:su3so3} for $SU(3)/SO(3)$. However, if they appear as irreducible factors of a reducible space $M_0$, we additionally know that $\Phi=0$ by Proposition~\ref{p:red}\eqref{it:red1}; then the claim of Proposition~\ref{p:main} is true, as we will show in Lemma~\ref{l:trace}\eqref{it:ltrace3} and in Section~\ref{s:rk1}.

\begin{proof}[Proof of Proposition~\ref{p:red}]
Let $M_0 = \prod_{s=1}^N M_s$ be the de Rham decomposition of $M_0$ on the irreducible symmetric spaces $M_s$ such that $T_oM_s=\m_s$.
Denote $R_s$ the curvature tensor of $M_s$ and $\pi_s:\m \to \m_s$ the orthogonal projection.

Choose $X,Y \in \m_s, \, Z, V \perp \m_s$ and act by the both sides of \eqref{eq:gbiad} on $V$. As for $s \ne r, \; [\m_s, \m_r]=[[\m,\m_r], \m_s]=0$ and $\<\m_s, \m_r\>=0$, we obtain after projecting to $\m_s$:
\begin{equation}\label{eq:subgbiad1}
[[X, Y],\pi_sK_ZV]+\<\Psi(Y,Z), V\> X + \<\Psi(Z,X), V\> Y -\<Z,V\> \pi_s \Psi(X,Y)=0.
\end{equation}
It follows that for any $Z, V \perp \m_s, \; Z \perp V$, the vector $T=\pi_sK_ZV \in \m_s$ satisfies $R_s(X,Y)T=(X \wedge Y)T'$, for any
$X, Y \in \m_s$, where $T' \in \m_s$ is defined by $\<T', X\>=\<\Psi(Z,X), V\>$, for $X \in \m_s$. Now, if $\rk M_s > 1$, we can take $X \in \m_s$ arbitrarily and then take $Y \in \m_s$ to commute with $X$ and to be nonproportional to $X$. This shows that $T'=0$, and therefore $R_s(X,Y)T=0$, for all $X, Y \in \m_s$, so $T=0$, as $M_s$ is irreducible and $\dim M_s >1$. Suppose now that $\rk M_s = 1$. Then from the fact that $M_s$ is Einstein, it follows that $T'=cT$, for some nonzero constant $c$, so $R_s(X,Y)T=c(X \wedge Y)T$. Moreover, assuming $T \ne 0$, we get $R_s(X,Y)Z=c(X \wedge Y)Z$, for all $X, Y, Z \in \m_s$, as the isotropy group acts transitively on the unit sphere of $\m_s$. It follows that $M_s$ has constant curvature, a contradiction. So in the both cases, $T=T'=0$, that is, for any $Z, V \perp \m_s$ with $Z \perp V$, we have $\pi_sK_ZV=0$ and $\<\Psi(Z,X), V\>=0$, for all $X \in \m_s$. By linearity, there exist $a_s, b_s \in \m_s$ such that for all $Z, V \perp \m_s$, we have $\pi_sK_ZV=\<Z,V\> a_s$, and for all $Z \perp \m_s, \; X \in \m_s$, $\Psi(Z,X)=\pi_s\Psi(Z,X)+\<b_s,X\>Z$. Substituting this to \eqref{eq:subgbiad1} we obtain $[[X, Y],a_s]+ \<b_s,X\>Y - \<b_s,Y\>X  -\pi_s \Psi(X,Y)=0$. Moreover, as $\Psi$ is skew-symmetric, we get from the above that $\Psi(Z,X)=-\<b_r,Z\>X+\<b_s,X\>Z$, for $X \in \m_s, \, Z \in \m_r, \; s \ne r$, so by \eqref{eq:cyclePsi}, $\pi_r \Psi(X,Y)=0$, for all $X,Y \in \m_s, \; r \ne s$. Thus there exist $a_s, b_s \in \m_s, \; s=1, \dots, N$, such that
\begin{equation}\label{eq:piKPsi}
\begin{split}
    \Psi(Z,X) &=-\<b_r,Z\>X+\<b_s,X\>Z,  \quad  X \in \m_s, \, Z \in \m_r, \; s \ne r,\\
    \Psi(X,Y) &= [[X, Y],a_s]+ \<b_s,X\>Y - \<b_s,Y\>X,  \quad  X, Y \in \m_s, \\
    \pi_sK_ZV &=\<Z,V\> a_s,  \quad  Z, V \perp \m_s,\\
    \pi_rK_ZX &= -\<a_s, X\>Z \quad  X \in \m_s, \, Z \in \m_r, \; s \ne r,
\end{split}
\end{equation}
where the fourth equation follows from the third one and the fact that $K_Z$ is skew-symmetric.

Now take $X,Y, V\in \m_s, \; Z  \in \m_r, \; r \ne s$, and act by the both sides of \eqref{eq:gbiad} on $V$. Projecting the resulting equation to $\m_s$ and using \eqref{eq:piKPsi} we get $[[X, Y],\pi_sK_ZV]-\pi_sK_Z[[X, Y],V]-[[\pi_sK_ZY, X],V]+[[\pi_sK_ZX, Y],V]+2\<a_r,Z\>[[X, Y],V]
+2\<b_r,Z\>\<Y, V\> X -2\<b_r,Z\>\<X, V\> Y =0$. Let $\cp^m(\m_s,\m_s), \; m \ge 0$, be the space of $m$-cochains of the Lie triple system $\m_s$ and let $\K: \cp^m(\m_s,\m_s) \to \cp^{m+2}(\m_s,\m_s)$ be the coboundary operator \cite{Yam}. We have $\cp^1(\m_s,\m_s)=\End(\m_s)$ and for $L \in \cp^1(\m_s,\m_s)$, the cochain $(\K L) \in \cp^3(\m_s,\m_s)$ is defined by $(\K L)(X_1,X_2,X_3)=-L[[X_1, X_2],X_3]+[[LX_1, X_2],X_3]+[[X_1, LX_2],X_3]+[[X_1, X_2],LX_3]$, for $X_j \in \m_s$. As $(\K \, \id_{\m_s})(X_1,X_2,X_3) = 2[[X_1, X_2],X_3]$, the above equation can be written as
\begin{equation}\label{eq:delta}
\K (\pi_sK_Z\pi_s+\<a_r,Z\>\id_{\m_s})=F, \quad \text{where} \quad F \in \cp^3(\m_s,\m_s), \quad F(X, Y, V)=2\<b_r,Z\>( X \wedge Y) V,
\end{equation}
for $X, Y, V \in \m_s$. It follows that $\K F=0$, for $\K: \cp^3(\m_s,\m_s) \to \cp^5(\m_s,\m_s)$. Using \cite[Eq.~(11)]{Yam} we obtain after simplification:
\begin{multline*}
0=(\K F)(X_1,X_2,X_3,X_4,X_5)=2\<b_r,Z\>(-(X_1\wedge X_2)[[X_3,X_4],X_5] \\
+[[(X_1\wedge X_2)X_3,X_4],X_5]+[[X_3,(X_1\wedge X_2)X_4],X_5]+[[X_3,X_4],(X_1\wedge X_2)X_5]),
\end{multline*}
for all $X_i \in \m_s, \; i=1, \dots, 5$. Assume $\<b_r,Z\> \ne 0$. As $M_s$ is not of constant curvature, we have $\dim M_s \ge 4$, so we can take linearly independent $X_1, X_2, X_5=X_3, X_4 \in \m_s$ such that $X_1, X_2 \perp X_3, X_4$, which gives $[[X_3,X_4],X_3] \in \Span(X_3, X_4)$, for all $X_3, X_4 \in \m_s$. But then $R_s(X,Y)X \parallel Y$, for all $X, Y \in \m_s, \; X \perp Y$, which easily implies that $M_s$ has constant curvature, a contradiction. It follows that $\<b_r,Z\>=0$, for all $Z \in \m_r$, so $b_r=0$. Then from \eqref{eq:delta}, the operator $\pi_sK_Z\pi_s+\<a_r,Z\>\id_{\m_s} \in \cp^1(\m_s,\m_s)=\End(\m_s)$ is a $1$-cocycle, that is, a derivation of $\m_s$. By \cite[Theorem~2.11]{Lis}, every derivation of $\m_s$ is inner, so there exists $P_{sr} \in \Hom(\m_r,\h_s)$ such that $\pi_sK_Z\pi_s+\<a_r,Z\>\id_{\m_s}= (\ad_{P_{sr}(Z)})_{\m_s}$. As both $K_Z$ and $\ad_{P_{sr}(Z)}$ are skew-symmetric, we obtain $a_r=0$. Now from \eqref{eq:piKPsi} and the fact that $a_r=b_r=0$ it follows that $\Psi=0$ and that $K_Z \m_s \subset \m_s$, for all $Z \in \m$ and all $s=1, \dots, N$. In particular, $[K_Z, \rho_0]=0$ (as every $M_s$ is Einstein), so $\Phi = 0$, by \eqref{eq:defPsi}. This proves assertion~\ref{it:red1}. Moreover, for every $Z \in \m_r, \; \pi_sK_Z\pi_s = (\ad_{P_{sr}(Z)})_{\m_s}$, which proves assertion~\ref{it:red2}.
\end{proof}

\section{Proof of Proposition~\ref{p:main}: irreducible case}
\label{s:irred}

By the arguments following the statement of Proposition~\ref{p:red} in Section~\ref{s:proof}, to prove Proposition~\ref{p:main} in full, we need to prove its claim for all irreducible symmetric spaces $M_0$ not of constant curvature, where in the cases when $n \le 5$, we may additionally assume that $\Psi=\Phi=0$. Note that a compact irreducible symmetric space of dimension $n \le 5$ of non-constant curvature is locally homothetic either to $SU(3)/SO(3) \; (n=5)$ or to $\bc P^2 \; (n=4)$.

We start with the following two observations.

First of all, in the irreducible case, the endomorphisms $\rho_0$ and $K_X$ commute, so from \eqref{eq:defPsi} $\Psi = \Phi$, hence equation \eqref{eq:gbiad} becomes
\begin{equation}\label{eq:biadphi}
    \sigma_{XYZ}([\ad_{[X, Y]},K_Z]+\ad_{[K_XY-K_YX, Z]}+ \Phi(X,Y)\wedge Z)=0,
\end{equation}
for all $X, Y, Z \in \m$, with $\Phi$ still satisfying \eqref{eq:cyclePhi}.

Secondly, if Proposition~\ref{p:main} is satisfied for a compact irreducible symmetric space, then it is also satified for its noncompact dual. Indeed, passing from $M_0$ to its dual effects in changing the sign of all the brackets $[X,Y],\; X, Y \in \m$, to the opposite. It follows that if a pair $(K, \Phi)$ satisfies \eqref{eq:biadphi} (and \eqref{eq:cyclePhi}) for a space $M_0$, then the pair $(K, -\Phi)$ satisfies the same equations for the dual space.

So we need to prove the following proposition.

\begin{proposition} \label{p:irre}
Let $M_0$ be a compact irreducible symmetric space of non-constant curvature. Let $o \in M_0$ and let $\m=T_oM_0$. Suppose that the maps $K \in \Hom(\m, \so(\m)), \; K: Z \mapsto K_Z$ and $\Phi \in \Hom(\Lambda^2\m, \m)$, $\Phi: X \wedge Y \mapsto \Phi(X,Y)$, satisfy \eqref{eq:cyclePhi} and \eqref{eq:biadphi}, for all $X,Y,Z \in \m$. If $n \le 5$ we additionally assume that $\Phi=0$. Then $K \in \Hom(\h, \ad(\h))$ and $\Phi=0$.
\end{proposition}

\begin{remark}\label{r:Kperpad}
Equation \eqref{eq:biadphi} is easily seen to be satisfied if $\Phi=0$ and $K_Z \in \ad(\h)$, for all $Z$. It follows that for any solution $(K,\Phi)$ of (\ref{eq:biadphi}, \ref{eq:cyclePhi}), $(\pi_{\ad(\h)^\perp}K, \Phi)$ is also a solution, where $\ad(\h)^\perp$ is the orthogonal complement to
$\ad(\h) \subset \so(\m)$. We can therefore assume for the rest of the proof that
\begin{equation}\label{eq:Kperpadh}
K_X \perp \ad(\h), \quad \text{for all} \quad X \in \m.
\end{equation}
and will be proving that the assumptions of Proposition~\ref{p:irre} together with \eqref{eq:Kperpadh} imply $\Phi=0$ and $K=0$.
\end{remark}

\begin{remark}\label{r:tg}
Equations (\ref{eq:cyclePhi}, \ref{eq:biadphi}) descend to an arbitrary Lie triple subsystem $\m' \subset \m$. Indeed, defining $K' \in \Hom(\m', \so(\m'))$ and $\Phi' \in \Hom(\Lambda^2\m', \m')$ by $K'_XY=\pi_{\m'}K_XY, \; \Phi'(X,Y)=\pi_{\m'}\Phi(X,Y)$, for $X, Y \in \m'$, where $\pi_{\m'}$ is the projection to $\m'$, we obtain that \eqref{eq:cyclePhi}, with $\Phi$ replaced by $\Phi'$ is trivially satisfied for all $X,Y,Z \in \m'$. Moreover, equation~\eqref{eq:biadphi} is equivalent to the fact that $\sigma_{XYZ}([-\<[[X, Y],V],K_ZU\>+\<[[X, Y],U],K_ZV\>+\<[[V,U], Y],K_ZX\>-\<[[V,U], X],K_ZY\>+ \<\Phi(X,Y),U\>\<Z, V\>-\<\Phi(X,Y),V\>\<Z, U\>)=0$, for all $X,Y,Z,U,V \in \m$. Taking all the vectors $X,Y,Z,U,V$ from $\m'$ and using the fact that $\m'$ is a Lie triple system, we obtain the same equation, with $K$ and $\Phi$ replaced by $K'$ and $\Phi'$ respectively.

Although the condition \eqref{eq:Kperpadh} may not always be satisfied for $K'$, we will be using the above observation as follows. If for
``sufficiently many" Lie triple subsystems $\m' \subset \m$ Proposition~\ref{p:irre} is satisfied, then $\<\Phi(X,Y),Z\>=0$ for sufficiently many
triples $X,Y,Z \in \m$ to imply $\Phi=0$.
\end{remark}

The proof of Proposition~\ref{p:irre} is based on the following technical facts:

\begin{lemma} \label{l:Amh}
Let $M_0$ be a compact irreducible symmetric space. Suppose that either $\rk M_0 \ge 2$ or $M_0=\Oc P^2$. Let a linear map $A: \m \to \h$ satisfy
\begin{equation*}
    \sigma_{XYZ}\<[X, Y],AZ\>=0,
\end{equation*}
for all $X, Y, Z \in \m$. Then there exists $T \in \m$ such that $A=\ad_T$.
\end{lemma}

\begin{lemma} \label{l:rkge3}
In the assumptions of Proposition~\ref{p:irre}, suppose that $\rk M_0 \ge 3$. Then
\begin{enumerate}[\rm 1.]
  \item \label{it:phi0}
  $\Phi(X, Y)=0$, for all $X, Y \in \m$ such that $[X, Y]=0$.
  \item \label{it:A[X,Y]}
  There exists a linear map $A:\h \to \m$ such that $\Phi(X, Y)=A[X, Y]$, for all $X, Y \in \m$.
\end{enumerate}
\end{lemma}

The proofs of Lemma~\ref{l:Amh} and Lemma~\ref{l:rkge3} will be given in Section~\ref{s:rr}.

\begin{lemma}\label{l:rk2}
Suppose that $M_0$ is a compact irreducible symmetric space of rank two other than $SU(3)/SO(3)$. In the assumptions of Proposition~\ref{p:irre}, $\Phi=0$.
\end{lemma}

Lemma~\ref{l:rk2} is proved in Section~\ref{s:rk2}.

The proof of Proposition~\ref{p:irre} for spaces $M_0$ of rank greater than one and for the Cayley projective plane is now completed by Lemma~\ref{l:rk2} and assertions~\ref{it:ltrace2} and \ref{it:ltrace3} of the following lemma.

\begin{lemma}\label{l:trace}
In the assumptions of Proposition~\ref{p:irre}, we have
\begin{enumerate}[\rm 1.]
  \item \label{it:ltrace1}
  $\sum_{i=1}^n \<\Phi(X,e_i),e_i\>=0$, for any $X \in \m$, where $e_i, \; i = 1, \dots, n$, is an orthonormal basis for $\m$.
  \item \label{it:ltrace2}
  If $\rk M_0 \ge 3$, then $\Phi=0$.
  \item \label{it:ltrace3}
  If $\rk M_0 \ge 2$ or $M_0=\Oc P^2$, then $K \in \Hom(\h, \ad(\h))$.
\end{enumerate}
\end{lemma}
\begin{proof}
1. For any $X \in \m$ and for any $K \in \so(\m)$ we have $\sum\nolimits_{i=1}^n [[X,e_i],e_i]=\tfrac12 X$ and $\sum\nolimits_{i=1}^n ([[X,Ke_i],e_i]+[[X,e_i],Ke_i])=0$ (the first identity is well known, the second one easily follows: the inner product of the left-hand side with an arbitrary $Y \in \m$ is $\Tr((\ad_Y\ad_XK)_{|\m})-\Tr((K\ad_Y\ad_X)_{|\m})=0$).

Now taking $Z=e_i$ in \eqref{eq:biadphi}, acting by the both sides on $e_i$ and then summing up by $i = 1, \dots, n$, we obtain using the above identities:
\begin{multline*}
    \sum\nolimits_{i=1}^n ([[K_{e_i}X,Y],e_i]+[[X,K_{e_i}Y],e_i]+[[X,Y],K_{e_i}e_i]-K_{e_i}[[X,Y],e_i])\\
    = (n-3)\Phi(X,Y) +\sum\nolimits_{i=1}^n (\<\Phi(X,e_i),e_i\>Y-\<\Phi(Y,e_i),e_i\>X).
\end{multline*}
Substituting $Y=e_j$, taking the inner product with $e_j$ and then summing up by $j = 1, \dots, n$, we obtain that $0=(2n-4) \sum_{i=1}^n\<\Phi(X,e_i),e_i\>$ and the claim follows, as $M_0$ is of non-constant curvature, so $n \ge 4$.

2. If $\rk M_0 \ge 3$, Lemma~\ref{l:rkge3}\eqref{it:A[X,Y]} implies the existence of $A:\h \to \m$ such that $\Phi(X, Y)=A[X, Y]$, for all $X, Y \in \m$. Then from \eqref{eq:cyclePhi} and Lemma~\ref{l:Amh} applied to $A^*$, the adjoint operator of $A$, we obtain that $\Phi(X, Y)=[T,[X, Y]]$, for some $T \in \m$. But then from assertion~\ref{it:ltrace1} we get $0=\sum_{i=1}^n\<[T,[X,e_i]],e_i\>=\frac12 \<T,X\>$, for all $X \in \m$. So $T=0$ and $\Phi=0$.

3. We have $\Phi=0$ (from the assumption of Proposition~\ref{p:irre} for $M_0=SU(3)/SO(3)$; from Section~\ref{s:rk1} for $M_0=\Oc P^2$; and from assertion~\ref{it:ltrace2} and Lemma~\ref{l:rk2} in all the other cases). By \eqref{eq:biadphi} we obtain $\sigma_{XYZ}([\ad_{[X, Y]},K_Z])+\sigma_{XYZ}(\ad_{[K_XY-K_YX, Z]})=0$. As by \eqref{eq:Kperpadh} we can assume that $K_X \perp \ad(\h)$, for all $X \in \m$, the first term on the left-hand side of \eqref{eq:biadphi} is also orthogonal to $\ad(\h)$, while the second term belongs to $\ad(\h)$, so for all $X, Y, Z \in \m$, we get
\begin{equation}\label{eq:twopieces}
    \sigma_{XYZ}[\ad_{[X, Y]},K_Z]=0, \qquad \sigma_{XYZ}[K_XY-K_YX, Z]=0.
\end{equation}
Acting by the first equation of \eqref{eq:twopieces} on an arbitrary $X_1 \in \m$ and then taking the inner product with $X_2 \in \m$ we obtain
$\sigma_{XYZ}\<[X, Y],[K_ZX_1,X_2]-[K_ZX_2,X_1]\>=0$, which by Lemma~\ref{l:Amh} implies the existence of $T=T(X_1,X_2) \in \m$ such that
$[K_ZX_1,X_2]-[K_ZX_2,X_1]=\ad_{T(X_1,X_2)}Z$. As the left-hand side is bilinear in $X_1, X_2$ and skew-symmetric, the same properties are satisfied
by the map $T$, so for all $X, Y, Z \in \m$,
\begin{equation}\label{eq:KT}
    [K_ZX,Y]-[K_ZY,X]=[T(X,Y),Z], \quad \text{for some } T \in \Hom(\Lambda^2\m, \m).
\end{equation}
Combined with the second equation of \eqref{eq:twopieces}, this implies $[K_XY+T(X,Y),Z]=[K_XZ+T(X,Z),Y]$. For every $X \in \m$, define $F_X \in \End(\m)$ by $F_XY=K_XY+T(X,Y)$. Then for all $Y,Z \in \m$ we have $[F_XY,Z]=[F_XZ,Y]$. Taking the inner product of the both sides with an arbitrary $U \in \h$ we obtain that $\ad_UF_X \in \End(\m)$ is symmetric, that is $\ad_UF_X=-F_X^t\ad_U$. By \cite[Lemma~4.2]{Sza}, we obtain that either $F_X=0$, or $M_0$ is Hermitian and $F_X$ is proportional to $J$, the complex structure on $\m$. As in the latter case $F_X$ depends linearly on $X$, it follows that
\begin{equation}\label{eq:0J}
    \text{either } T(X,Y)=-K_XY, \quad \text{or for some }l \in \m, \; T(X,Y)=-K_XY + \<l,X\>JY.
\end{equation}
Note that in the both cases, $\<T(X,Y),Y\>=0$, so as $T(X,Y)=-T(Y,X)$,
the trilinear map $(X,Y,Z) \mapsto \<T(X,Y),Z\>$ is skew-symmetric. Moreover, from the second equation of
\eqref{eq:twopieces} and from \eqref{eq:KT} we obtain $\sigma_{XYZ}[T(X,Y),Z]=0$. Taking the inner product with an arbitrary $U \in \h$ and using
the skew-symmetry of $\<T(X,Y),Z\>$ we obtain $[U,T(X,Y)]=T([U,X],Y)+T(X,[U,Y])$, so defining for every $X \in \m$ the operator
$T_X$ by $T_XY=T(X,Y)$ we get $T_{[U,X]}=[\ad_U,T_X]$, for all $X \in \m, \; U \in \h$. Then for an orthonormal basis $\{U_a\}$ for $\h$ we get
\begin{equation*}
    \sum\nolimits_a T_{[U_a,[U_a,X]]}=\sum\nolimits_a[\ad_{U_a},[\ad_{U_a},T_X]]=\sum\nolimits_a(\ad_{U_a}^2T_X-2\ad_{U_a}T_X\ad_{U_a}+T_X\ad_{U_a}^2).
\end{equation*}
As it is well known, $\sum_a[U_a,[U_a,X]]=\frac12X$, and then $\sum_a\ad_{U_a}^2T_X=\sum_aT_X\ad_{U_a}^2=\frac12T_X$. Moreover, for any $Y,Z \in \m$,
we have $\<(\sum_a\ad_{U_a}T_X\ad_{U_a})Y,Z\>=\Tr((\ad_ZT_X\ad_Y)_{|\h})=\Tr((T_X\ad_Y\ad_Z)_{|\m})$. so the above equation gives
$\<T_XY,Z\>=4\Tr((T_X\ad_Y\ad_Z)_{|\m})$. Subtracting the same equation, with $Y$ and $Z$ interchanged and using \eqref{eq:0J} and the fact that
$K_X \perp \ad(\h)$, we obtain that $T=0$ (and hence $K=0$) in the first case of \eqref{eq:0J} and that $\<T(X,Y),Z\>=4\<l,X\>\Tr((J\ad_{[Y,Z]})_{|\m})$, in the second case. If $l \ne 0$, the skew-symmetry of $T$ implies $\Tr((J\ad_{[Y,X]})_{|\m})=0$, for all $X, Y \in \m$, so $\Tr((J \ad_{U})_{|\m})=0$, for all $U \in \h$. But for a Hermitian symmetric space $M_0$, we have $J=\ad_{U_0}$, where $U_0$ spans the center of $\h$, so $\Tr((J \ad_{U_0})_{|\m})=\Tr (J^2)=-n$, a contradiction. It follows that $l=0$, so $T=0$ and $K=0$, also in this case.
\end{proof}

\begin{remark}\label{rem:KPhi}
Note that by equation~\eqref{eq:biadphi}, $\Phi$ is uniquely determined by $K$, namely, from the equation obtained in the proof of  Lemma~\ref{l:trace}\eqref{it:ltrace1} and the fact that $\sum_{i=1}^n \<\Phi(X,e_i),e_i\>=0$ it follows that $\sum_{i=1}^n(\K K_{e_i})(X,Y,e_i)=(n-3)\<\Phi(X,Y),Z\>$ (in the notation of Proposition~\ref{p:red}). Moreover, equation \eqref{eq:cyclePhi} is then automatically satisfies.
\end{remark}

The proof of Proposition~\ref{p:irre} in the remaining cases, for the complex and the hyperbolic projective spaces, and also of the fact that $\Phi=0$ for the Cayley projective plane which was used in the proof of Lemma~\ref{l:trace}\eqref{it:ltrace3} is given in Section~\ref{s:rk1}.

\section{Proof of Lemma~\ref{l:Amh} and Lemma~\ref{l:rkge3}}
\label{s:rr}

We start with briefly recalling some well-known facts on the restricted roots (see \cite{Hel, Nag,NT}).

Let $\ag \subset \m$ be a Cartan subspace (a maximal abelian subspace) and let $\Delta \subset \ag^*$ be the set of restricted roots. We have
orthogonal root decompositions
\begin{equation*}
    \m=\ag \oplus \bigoplus\nolimits_{\a \in \Delta^+}\m_\a, \qquad \h=\h_0 \oplus \bigoplus\nolimits_{\a \in \Delta^+}\h_\a,
\end{equation*}
where $\Delta^+$ is the set of positive roots, and $\m_\a=\m_{-\a}, \; \h_\a=\h_{-\a}$ are the root spaces. For any $\a \in \Delta$ there exists a
linear isometry $\theta_\a:\m_\a \to \h_\a, \; \theta_{-\a}=-\theta_\a$, such that $[H, X_\a]=\a(H)\theta_\a X_\a, \; [H, \theta_\a X_\a]=-\a(H) X_\a$, for all $H \in \ag, \; X_\a \in \m_\a$ (in particular, $\dim \m_\a=\dim \h_\a =: m_\a$, the multiplicity of the root $\a \in \Delta$). Moreover, for any $\a, \b \in \Delta \cup \{0\}, \; [\m_\a,\m_\b], [\h_\a,\h_\b] \subset \h_{\a+\b}+\h_{\a-\b}, \; [\h_\a,\m_\b] \subset \m_{\a+\b}+\m_{\a-\b}$, where we denote $\m_0=\ag$. A root system $\Delta$ is called \emph{reduced}, if $\Delta \cap \br \a = \pm \a$, for any $\a \in \Delta$.

\begin{lemma}\label{l:rr}
{ \ }
\begin{enumerate}[\rm 1.]
  \item \label{it:rr1}
  If the roots $\a$ and $\b$ are not proportional and not orthogonal, then $[X_\a,X_\b] \ne 0$ and $[X_\a,\theta_\b X_\b] \ne 0$, for all nonzero $X_\a \in \m_\a, \; X_\b \in \m_\b$.
  \item \label{it:rr2}
  If $\Delta$ is reduced and $\a, \b \in \Delta$, with $\<\a, \b\> > 0, \a \ne \b$, then $[\m_\a,\m_\b]$ has a nonzero $\h_{\a-\b}$ component. \item \label{it:rr3}
  If $\a, \b, \gamma=\a+\b \in \Delta$ and $\a$ and $\b$ are not proportional, then $\<[X_\gamma, X_\a], \theta_\b X_\b\>=\<[X_\b, X_\gamma], \theta_\a X_\a\>$ $=-\<[X_\a, X_\b], \theta_\gamma X_\gamma\>$, for all $X_\a \in \m_a, \; X_\b \in \m_\b$ and $X_\gamma \in \m_\gamma$.
\end{enumerate}
\end{lemma}
\begin{proof}
1. Suppose that $[X_\a,X_\b] = 0$ and let $H \in \ag$ be such that $\a(H) \ne 0, \; \b(H)=0$. Then $0=[H,[X_\a, X_\b]]=\a(H)[X_\a, \theta_\b X_\b]$, so $[X_\a, \theta_\b X_\b]=0$. Similarly, if $[X_\a, \theta_\b X_\b]=0$, then $[X_\a,X_\b] = 0$. In the both cases, we get
$0=[X_\a,[X_\b,\theta_\b X_\b]]$. As $[X_\b,\theta_\b X_\b] \in \|X_\b\|^2\b^* + \m_{2\b}$, where $\b^* \in \ag$ is dual to $\b$, we obtain
$0=[X_\a,[X_\b,\theta_\b X_\b]]=-\|X_\b\|^2 \<\a, \b\>\theta_\a X_\a$ plus possibly some element from $\m_{2\b-\a}+\m_{2\b+\a}$, a contradiction.

2. The fact that $\a-\b \in \Delta$ when $\<\a, \b\> > 0, \a \ne \b$, is a general property of a root system. The subset $\Delta'=(\mathbb{Z}\a+ \mathbb{Z}\b) \cap \Delta$ is a root subsystem of $\Delta$ of type $\mathrm{A}_2, \mathrm{B}_2$ or $\mathrm{G}_2$. 
In the first two cases, for any two roots $\a',\b' \in \Delta'$ with $\<\a',\b'\> > 0$, we have $\a'+\b' \notin \Delta'$, so $\a+\b \notin \Delta$ and the claim follows from assertion~\ref{it:rr1}. In the third case, the same argument applies, unless all three roots $\a', \b', \a'-\b'$ are short. But then the subspace $\m'=\Span(\a^*,\b^*) \oplus \sum_{\gamma \in \Delta} \m_\gamma$ is a Lie triple subsystem of $\m$ tangent to a compact symmetric space with the root system $\mathrm{G}_2$, that is, either to the group $G_2$ or to $G_2/\mathrm{SO}(4)$. As the latter space has the maximal rank, the claim follows from the fact that for any three short roots $\a_1, \b_1, \a_1-\b_1$ of the complex simple Lie algebra $\g_2^\bc$, we have $[\g_{\a_1},\g_{-\b_1}]=\g_{\a_1-\b_1}$ for the corresponding root spaces.

3. For $H \in \ag$ such that $\a(H)=0, \; \b(H) \ne 0$, we have $\b(H)\<[X_\a, X_\b], \theta_\gamma X_\gamma\>=\<[X_\a, X_\b], [H, X_\gamma]\>=
-\<[X_\a, [H, X_\b]], X_\gamma\>=-\b(H)\<[X_\a, \theta_\b X_\b], X_\gamma\>=-\b(H)\<[X_\gamma, X_\a], \theta_\b X_\b\>$, hence
$\<[X_\gamma, X_\a], \theta_\b X_\b\>$ $=-\<[X_\a, X_\b], \theta_\gamma X_\gamma\>$. Interchanging $\a$ and $\b$ we get the second equation.
\end{proof}

We will also use the following elementary fact of linear algebra.
\begin{lemma}\label{l:cross}
Let $V$ be a complex or a real Euclidean space, and let $\Psi: \Lambda^2V \to V$ be a linear map.
\begin{enumerate}[\rm 1.]
  \item \label{it:cr1}
  If $\Psi(X,Y) \in \Span(X,Y)$, for all $X, Y \in V$, then $\Psi(X,Y)=\<p,X\>Y-\<p,Y\>X$, for some $p \in V$.
  \item \label{it:cr2}
  If $\sigma_{XYZ}(\Psi(X,Y) \wedge Z)=0$, for all $X,Y,Z \in V$ and $\dim V \ge 4$, then $\Psi=0$.
\end{enumerate}
\end{lemma}

\begin{proof}
1. Relative to an orthonormal basis $e_i$ for $V$ we have $\Psi(e_i,e_j)=a_{ij}e_i-a_{ji}e_j$. Then by linearity, $a_{ij}=a_{kj}$ for all $k,i \ne j$. Take $p_i=-a_{ji}, \; j \ne i$.

2. Taking the inner product of the $\sigma_{XYZ}(\Psi(X,Y) \wedge Z)=0$ with $W \perp X, Y, Z$ (where $X,Y,Z$ are linearly independent) we get $\<\Psi(X,Y),W\>=0$, so $\Psi(X,Y)\in \Span(X,Y,Z)$, so $\Psi(X,Y)\in \Span(X,Y)$, for all $X, Y \in V$. By assertion~\ref{it:cr1}, $\Psi(X,Y)=\<p,X\>Y-\<p,Y\>X$ for some $p \in V$. But then $0=\sigma_{XYZ}(\Psi(X,Y) \wedge Z)=2\sigma_{XYZ}(\<p,X\>Y \wedge Z)$, so $p=0$.
\end{proof}

We now prove Lemma~\ref{l:rkge3} and Lemma~\ref{l:Amh} from Section~\ref{s:irred}.

\begin{replemma}{l:rkge3}
In the assumptions of Proposition~\ref{p:irre}, suppose that $\rk M_0 \ge 3$. Then
\begin{enumerate}[\rm 1.]
  \item 
  $\Phi(X, Y)=0$, for all $X, Y \in \m$ such that $[X, Y]=0$.
  \item 
  There exists a linear map $A:\h \to \m$ such that $\Phi(X, Y)=A[X, Y]$, for all $X, Y \in \m$.
\end{enumerate}
\end{replemma}

\begin{proof}
1. Let $\ag \subset \m$ be a Cartan subspace. Substituting $X,Y,Z \in \ag$ into \eqref{eq:biadphi} we obtain
\begin{equation}\label{eq:rkge3}
\ad_U=-\sigma_{XYZ}(\Phi(X,Y)\wedge Z),
\end{equation}
where $U=\sigma_{XYZ}[K_XY-K_YX, Z] \in [\ag, \m] \subset \h$.

We first prove the assertion under the assumption that $\rk M_0 \ge 4$. Chose a regular element $V \in \ag$ and then a three-dimensional subspace $\ag_3 \subset (\ag \cap V^\perp)$. The set of such subspaces $\ag_3$ is open in the Grassmannian $G(3, \ag)$. Taking $X,Y,Z$ in \eqref{eq:rkge3} spanning the subspace $\ag_3$ and acting by the both sides on $V$ we obtain that $[U,V] \in \ag_3$, as $V \perp \ag_3$. But $\<[U,V],X\>=\<U,[V,X]\>=0$ (and similarly for $Y$ and $Z$), so $[U,V]=0$. As $V \in \ag$ is regular, it follows that $[U, \ag]=0$, so $\ad_U \ag_3=0$. But from \eqref{eq:rkge3} we have $\<\ad_U \ag_3^\perp, \ag_3^\perp\>=0$. As $\ad_U$ is skew-symmetric, we obtain $\ad_U=0$, so $\sigma_{XYZ}(\Phi(X,Y)\wedge Z)=0$ by \eqref{eq:rkge3}. Taking the inner product of this equation with any vector from $\ag^\perp$ and using the fact that $X, Y,Z$ are linearly independent, we obtain that $\Phi(X,Y) \in \ag$. As this is satisfied for all $\ag_3$ from an open subset of the Grassmannian $G(3, \ag)$, we get that $\Phi(X,Y) \in \ag$, for any $X,Y \in \ag$. Then by Lemma~\ref{l:cross}\eqref{it:cr2}, $\Phi(X,Y)=0$, for all $X, Y \in \ag$. This proves the assertion as any two commuting elements of $\m$ lie in a Cartan subspace.

Now suppose that $\rk M_0 = 3$. Let $\ag \subset \m$ be a Cartan subspace. For any $X,Y,Z$ spanning $\ag$, equation \eqref{eq:rkge3} gives $[U,\ag^\perp] \subset \ag$, so
\begin{equation}\label{eq:rk3ad}
[U,\m_\b] \subset \ag, \text{ for any } \b \in \Delta.
\end{equation}
Note that $U \in [\ag, \m] = \oplus_{\a \in \Delta^+}\h_\a$, so $U =\sum_{\a \in \Delta^+}U_\a$, for some $U_\a \in \h_\a$. It follows from \eqref{eq:rk3ad} that for any two nonproportional roots $\a, \b$ we have $\sum_{j \in \mathbb{Z}}[U_{\a+j\b},\m_\b]=0$, where the sum is taken over all $j \in \mathbb{Z}$ such that $\a+j\b \in \Delta$. In particular, if the $\b$-series of $\a$ has length two and $\a \not\perp \b$, we obtain $U_\a=0$ by Lemma~\ref{l:rr}\eqref{it:rr1}. Now, from the classification of restricted root systems (see \cite{Hel} or the table in \cite{Tam}), we get that $\Delta$ is of one of types $\mathrm{A}_3, \mathrm{B}_3, \mathrm{C}_3, \mathrm{D}_3$ or $\mathrm{BC}_3$. As for the root systems of types $\mathrm{A}_3, \mathrm{B}_3, \mathrm{D}_3$, every root $\a$ can be included in a $\b$-series of length two, with $\a \not\parallel \b, \; \a \not\perp \b$, we obtain $U_\a=0$, for all $\a \in \Delta$, that is, $U=0$. If $\Delta$ is of type $\mathrm{BC}_3$ (so that $\Delta=\{\pm \omega_i, \pm 2\omega_i, \pm \omega_i \pm \omega_j\},\; 1 \le i < j \le 3$), the same arguments work for all the roots except for $\pm 2\omega_i$. But if $U=\sum_{i=1}^3 U_{2\omega_i}$, equation \eqref{eq:rk3ad} implies that $[U_{2\omega_i}, \m_i]=0$, for all $i=1,2,3$. It then follows that $U_{2\omega_i}=0$, as $\br\omega_i^* \oplus \m_{\omega_i} \oplus \m_{2\omega_i}$ is a Lie triple system tangent to a rank one symmetric space (actually, to a complex or to a quaternionic projective space). Hence $U_\a=0$, for all $\a \in \Delta$, so $U=0$ in this case, as well. Finally,
suppose that $\Delta$ is of type $\mathrm{C}_3$ (so that $\Delta=\{\pm \omega_i \pm \omega_j, \pm 2\omega_i \},\; 1 \le i < j \le 3$). As every root
$\pm \omega_i \pm \omega_j, \; i \ne j$, is a member of the $2\omega_i$-series of length two, the same arguments as above show that $U_{\pm \omega_i \pm \omega_j}=0$, so $U=\sum_{i=1}^3 U_{2\omega_i}$. Then by \eqref{eq:rk3ad}, $[U, \m_{\pm \omega_i \pm \omega_j}]=0,\; i \ne j$. As $U$ commutes with both $\m_{\omega_1 + \omega_3}$ and $\m_{\omega_2 - \omega_3}$, it also commutes with $[\m_{\omega_1 + \omega_3},\m_{\omega_2 - \omega_3}]=\h_{\omega_1 + \omega_2}$ (the equality follows from Lemma~\ref{l:rr}\eqref{it:rr1}). Therefore $U$ commutes with the subspace $[\h_{\omega_1 + \omega_2},\m_{\omega_1 + \omega_2}]=\br (\omega_1 + \omega_2)^*$. It follows that $U_{2\omega_1}=U_{2\omega_2}=0$. Similar argument shows that also $U_{2\omega_3}=0$, hence again $U=0$.

As $U=0$ in all the cases, equation \eqref{eq:rkge3} implies that $\sigma_{XYZ}(\Phi(X,Y)\wedge Z)=0$, for all $X,Y,Z$ spanning a Cartan subspace
$\ag \subset \m$. Taking the inner product of this equation with any vector from $\ag^\perp$, we obtain that $\Phi(X,Y) \in \ag$, for all
$X,Y \in \ag$. Then from $\sigma_{XYZ}(\Phi(X,Y)\wedge Z)=0$ it follows that for every Cartan subspace $\ag$ there exists a symmetric operator
$S^{\ag} \in \Sym(\ag)$ such that $\Phi(X,Y)=S^{\ag}(X \times Y)$, where $X \times Y$ is the cross-product in the three-dimensional Euclidean space
$\ag$. Now, for every root $\a \in \Delta$, the subspace $\ag'=\Ker \a \oplus \br X_\a$, with a nonzero $X_\a \in \m_\a$, is again a Cartan subspace,
so for $X, Y \in \Ker \a, \quad \Phi(X,Y) \in \ag \cap \ag'=\Ker \a$. It follows that $S^{\ag} \a^* \perp \a^*$, for any $\a \in \Delta$, that is,
$\<S^{\ag} \a^*, \a^*\>=0$. An inspection of root systems of types $\mathrm{A}_3, \mathrm{B}_3, \mathrm{C}_3, \mathrm{D}_3$ and $\mathrm{BC}_3$ shows that this implies $S^{\ag}=0$ in all the cases. Therefore $\Phi(X,Y)=0$, for all $X, Y \in \ag$. This proves the assertion also for the spaces of rank three.

2. It suffices to prove the following: if $K \in \so(\m)$ is a skew-symmetric operator such that $\<KX,Y\>=0$, for any $X, Y \in \m$ with $[X,Y]=0$,
then there exists $U \in \h$ such that $K=\ad_U$ (indeed, by assertion~1, we would then have that for every $e \in \m$, there exists $U \in \h$
such that $\<\Phi(X,Y),e\>=\<\ad_UX,Y\>=\<U, [X,Y]\>$).

Introduce the boundary operator $\db:\so(\m) \to \h$ by putting $\db(X \wedge Y) = [X, Y]$ and extending by linearity (it is easy to see that $\db$
is well-defined and that for $K \in \so(\m), \; \db(K)=-\frac12 \sum_i [Ke_i,e_i]$, for an orthonormal basis $e_i$ of $\m$).
The space $\so(\m)$ is an $\h$-module, with an $\h$-invariant inner product $\<A_1,A_2\>=\Tr(A_1A_2^t)$. For $K \in \so(\m), \; X, Y \in \m$, we have
$\<K, X \wedge Y\>=2\<KX,Y\>$. In particular, for $U \in \h, \; X, Y \in \m$, we obtain $\<\ad_U, X \wedge Y\>=2\<U, [X,Y]\>$, so the orthogonal
complement to the submodule $\ad(\h) \subset \so(\m)$ is an $\h$-module $\mathcal{M} = \Ker \db$, the space of all those
$K = \sum_i X_i \wedge Y_i \in \so(\m), \; X_i, Y_i \in \m$, such that $\sum_i [X_i,Y_i]=0$.
The fact that $\<KX,Y\>=0$, for any $X, Y \in \m$ with $[X,Y]=0$, is equivalent to the fact that $K$ is orthogonal to the subspace
$\mathcal{D} \subset \mathcal{M}$ spanned by all $X \wedge Y \in \mathcal{M}$ (we will call the elements of $\mathcal{D}$ \emph{decomposable}).

The claim of the assertion is therefore equivalent to the fact that every element of the $\h$-submodule $\mathcal{M}$ is decomposable, that is, to the fact that if $K = \sum_i X_i \wedge Y_i, \; X_i, Y_i \in \m$, with $\sum_i [X_i,Y_i]=0$, then $K = \sum_j X'_j \wedge Y'_j, \; X'_j, Y'_j \in \m$,
with $[X'_j,Y'_j]=0$, for every $j$. Clearly, $\mathcal{D} \subset \mathcal{M}$ is an $\h$-submodule, as for any $U \in \h$ and for any commuting
$X, Y \in \m$, we have
\begin{equation}\label{eq:UXYXUY}
\mathcal{D} \ni \frac{d}{dt}_{|t=0}(\exp(t\ad_U)X)\wedge(\exp(t\ad_U)Y)=[U,X] \wedge Y + X \wedge [U,Y]=[\ad_U,X \wedge Y].
\end{equation}
Let $\ag \subset \m$ be a Cartan subspace and $\m_\a, \h_\a$ be the root subspaces. We will use the following facts:

\smallskip

\emph{Fact 1.} If $X_1 \in \m_a$ and $X_2 \in \m_\b, \; \b \nparallel \a$, then $X_1 \wedge X_2 \equiv H \wedge Z \mod(\mathcal{D})$, where
$H \in \ag, \; Z \in \m_{\a+\b} \oplus \m_{\a-\b}$. To see that, choose $H \in \ag$ such that $\a(H)=0 \ne \b(H)$. Then $[H, X_1]=0$, so
by \eqref{eq:UXYXUY} with $U=\b(H)^{-1}\theta_bX_2, \; X=X_1, \; Y=H$, we obtain $X_1 \wedge X_2-H \wedge [\b(H)^{-1}\theta_bX_2, X_1] \in \mathcal{D}$.

\smallskip

\emph{Fact 2.} If $K=\sum_i H_i \wedge X_i \in \mathcal{M}$, where $H_i \in \ag$, then $K \in \mathcal{D}$. Indeed, as
$\ag\wedge \ag \subset \mathcal{D}$ and $H \wedge X \in \mathcal{D}$, if $H \in \ag, \; X \in \m_\a$ and $\a(H)=0$, we obtain that
$K \equiv \sum_{\a \in \Delta^+} \a^* \wedge X^{\a}\mod(\mathcal{D})$, where $X^{\a} \in \m_\a$. But the latter sum belongs to $\mathcal{M}$, only
if all the $X^\a$ are zero (as $\db(\a^* \wedge X^{\a}) \in \m_\a$).

\smallskip

\emph{Fact 3.} The claim of the assertion (which is equivalent to the fact that $\mathcal{D}=\mathcal{M}$) is equivalent to the fact that
$[\ad_U, K] \in \mathcal{D}$, for any $K \in \mathcal{M}$ and any $U \in \h_\a,\; \a \in \Delta^+$.

Indeed, although the $\h$-module $\mathcal{M}$ can be reducible, it contains no trivial submodules, that is, no nonzero $K \in \mathcal{M}$ commutes with $\ad(\h)$. Otherwise, for such a $K$ we would have had $K\ad_{[X_1,X_2]}X_3=\ad_{[X_1,X_2]}KX_3$, so
$\<[[X_1,X_2],X_3],KX_4\> + \<[[X_1,X_2],KX_3],X_4\>=0$, for all $X_1,X_2,X_3,X_4 \in \m$, so
$\<[[KX_1,X_2],X_3],X_4\> + \<[[X_1,KX_2],X_3],X_4\>+\<[[X_1,X_2],X_3],KX_4\> + \<[[X_1,X_2],KX_3],X_4\>=0$, which would imply that $K$ is a derivation of the Lie triple system $\m$, so $K \in \ad(\h)=\mathcal{M}^\perp$ \cite[Theorem~2.11]{Lis}.

It follows that for any $K \in \mathcal{M}$, there exist $U_i \in \h$ and $K_i \in \mathcal{M}$ such that $K=\sum_i[\ad_{U_i}, K_i]$, so the claim of
the assertion is equivalent to the fact that $[\ad_U, K] \in \mathcal{D}$, for any $U \in \h$ and any $K \in  \mathcal{M}$. Suppose that we can prove
this fact for any $K \in  \mathcal{M}$ and any $U \in \h_\a,\; \a \in \Delta^+$. Then, as $\mathcal{D}$ is an $\h$-module, we obtain that $[\h_\a,[\h_\a,K]] \in \mathcal{D}$, for all $\a \in \Delta^+$, hence $[[\h_\a,\h_\a],K] \in \mathcal{D}$. But $\sum_{\a \in \Delta^+} [\h_\a,\h_\a] + \sum_{\a \in \Delta^+} \h_\a \supset \h_0$ (and in the reduced case, we simply have $\sum_{\a \in \Delta^+} [\h_\a,\h_\a] = \h_0$).
Otherwise, $\pi_{\h_0} \sum_{\a \in \Delta^+} [\h_\a,\h_\a] \ne \h_0$, so there exists a nonzero $U \in \h_0$ orthogonal to all the $[\h_\a,\h_\a]$.
But then $[U, \h_\a]=0$, and from $[U, \ag]=0$ it follows that $[U, \m_\a]=0$, so $[U, \m]=0$, hence $U=0$. Therefore $[\ad_{\h_0},K] \in \mathcal{D}$, hence $[\ad_\h,K] \in \mathcal{D}$, which proves the claim of Fact 3.

\medskip

In view of Fact 3, we have to prove that $[\ad_U, K] \in \mathcal{D}$, for any $K \in \mathcal{M}$ and any $U \in \h_\a,\; \a \in \Delta^+$. This is
trivially satisfied, if $K$ by itself belongs to $\mathcal{D}$, as $\mathcal{D}$ is an $\h$-module. Given $K \in \mathcal{M}$, it can be represented as $K=\sum_i X_i \wedge Y_i$, where every $X_i$ and every $Y_i$ belongs either to $\ag$, or to some $\m_\a$. By Fact 1, we can assume that $K= \sum_{\a \in \Delta^+} (\sum_i H_i^{\a} \wedge X_i^{\a}+ \sum_j Y_j^{\a} \wedge Z_j^{\a})$, where $H_i^{\a} \in \ag, \; X_i^{\a}, Y_j^{\a}, Z_j^{\a} \in \m_\a$.

First suppose that $\Delta$ is reduced. Then, as $K \in \mathcal{M}$, we have $\db(\sum_i H_i^{\a} \wedge X_i^{\a}) \in \h_\a$  and $\db(\sum_j Y_j^{\a} \wedge Z_j^{\a}) \in \h_0$, so $\sum_{\a \in \Delta^+} (\sum_i H_i^{\a} \wedge X_i^{\a}) \in \mathcal{M}$, hence it belongs to $\mathcal{D}$, by Fact 2. It remains to show that $[\ad_U,K] \in \mathcal{D}$, for any $U \in \h_\b, \; \b \in \Delta^+$, and any $K \in \mathcal{M}$ of the form $K=\sum_{\a \in \Delta^+} (\sum_j Y_j^{\a} \wedge Z_j^{\a})$, which follows from Fact 1 and Fact 2, as $[\ad_U, K]$ is a sum of the terms of the form $Y \wedge Z, \; Y \in \m_\a, Z \in \m_{\a\pm\b}$ (or $Z \in \ag$, if $\a = \pm \b$).

Now suppose that $\Delta$ is non-reduced, that is, $\Delta$ is of type $\mathrm{BC}_r$, so that $\Delta=\{\pm \omega_i, \pm 2\omega_i, \pm \omega_i \pm \omega_j\},\; 1 \le i < j \le r$. Let $K \in \mathcal{M}$. Using Fact 1 we can assume that $K$ is a linear combination of $X \wedge Y$ such that either $X \in \ag, \; Y \in \m_\a, \; \a \in \Delta$, or $X, Y \in \m_\a, \; \a \in \Delta$, or $X \in \m_{\omega_i}, \; Y \in \m_{2\omega_i}$.

The only terms $X \wedge Y$ in $K$ such that $\db(X \wedge Y) \in \h_{\omega_i+\ve \omega_j}, \; i \ne j, \; \ve = \pm 1$, are the terms with
$X \in \ag, \; Y \in \m_{\omega_i+\ve \omega_j}$. As $K \in \mathcal{M}$, the sum of all these terms appearing in $K$ also belongs to $\mathcal{M}$,
hence to $\mathcal{D}$, by Fact 2. The only terms $X \wedge Y$ in $K$ such that $\db(X \wedge Y) \in \h_{\omega_i}$, are the terms with
$X \in \ag \cup \m_{2\omega_i}, \; Y \in \m_{\omega_i}$. As $K \in \mathcal{M}$, the sum of all these terms appearing in $K$ also belongs to
$\mathcal{M}$. This sum has the form $K_i=\sum_a H_a \wedge Y_a + \sum_b X_b \wedge Y_b, \; H_a \in \ag, \; Y_a, Y_b \in \m_{\omega_i}, \;
X_b \in \m_{2\omega_i}$. Consider a term $X_b \wedge Y_b$ from the second sum. Let $j \ne i$ and let $U \in \h_{\omega_i+\omega_j}$ be nonzero. By Lemma~\ref{l:rr}\eqref{it:rr1} and as $\dim \m_{\omega_i}=\dim \m_{\omega_j}$, the map $\ad_U:\m_{\omega_j} \to \m_{\omega_i}$ is surjective, so there exists $Z_b \in \m_{\omega_j}$ with $[U,Z_b]=Y_b$. As $[X_b,Z_b]=0$, we obtain by \eqref{eq:UXYXUY}, with $X=X_b, \; Y=Z_b$, that
$X_b \wedge Y_b \equiv -[U,X_b] \wedge Z_b \mod(\mathcal{D})$, hence by Fact 1 (as $[U,X_b] \in \m_{\omega_i-\omega_j}$),
$X_b \wedge Y_b \equiv  H_b \wedge Y'_b \mod(\mathcal{D})$, for some $H_b \in \ag, \; Y'_b \in \m_{\omega_i}$. It follows that
$K_i \equiv \sum_c H_c \wedge Y_c \mod(\mathcal{D}), \; H_c \in \ag, \; Y_c \in \m_{\omega_i}$. As $K_i \in \mathcal{M}$, we obtain that
$K_i \in \mathcal{D}$, by Fact 2.

Therefore, we can assume that $K \in \mathcal{M}$ is a linear combination of the terms $X \wedge Y$ such that either $X \in \ag, \; Y \in \m_{2\omega_i}$, or $X, Y \in \m_\a, \; \a \in \Delta$. In view of Fact 3, it suffices to prove that $[\ad_U, K] \in \mathcal{D}$, for any such $K$ and any $U \in \h_\b,\; \b \in \Delta^+$. Now, if $\b=\omega_i \pm \omega_j, \; i \ne j$, then $[\ad_U, K]$ is a sum of the terms of the form $X' \wedge Y', \; X' \in \m_\gamma, Y' \in \m_\delta, \; \gamma \nparallel \delta$ (or $X' \in \ag, \; Y' \in \m_\delta$). This sum still belongs to $\mathcal{M}$, as the latter is an $\h$-module, hence we are done by applying Fact 1 and then Fact 2. Now suppose that $\b=\omega_i$ or $\b=2\omega_i$. Then the same arguments still work, provided we can show that $K \equiv K' \mod(\mathcal{D})$, where $K' \in \mathcal{M}$ is a linear combination of the terms $X \wedge Y$ such that either $X \in \ag, \; Y \in \m_{2\omega_j}, \; j \ne i$, or $X, Y \in \m_\a, \; \a \in \Delta^+ \setminus \{\omega_i, 2\omega_i\}$. To see that, suppose that $K$ contains a term $X \wedge Y, \; X,Y \in \m_{\omega_i}$. Choose $j \ne i$. By Lemma~\ref{l:rr}\eqref{it:rr1} and as $\dim \m_{\omega_i}=\dim \m_{\omega_j}$, the map $\ad_V:\m_{\omega_j} \to \m_{\omega_i}$ is surjective for any nonzero $V \in \h_{\omega_i+\omega_j}$, so there exists $V \in \h_{\omega_i+\omega_j}, \; Z \in \m_{\omega_j}$ such that $[V,Z]=Y$. Now $[X,Z] \in \h_{\omega_i+\omega_j} \oplus \h_{\omega_i-\omega_j}$, so there exist $Z_{\pm} \in \m_{\omega_i\pm\omega_j}$ such that $[X,Z]+[H_+,Z_+]+[H_-,Z_-]=0$, where $H_\pm=(\omega_i+\omega_j)^* \in \ag$, that is, $X \wedge Z + H_+ \wedge Z_+ + H_- \wedge Z_-$ belongs to $\mathcal{M}$, hence to $\mathcal{D}$, by Fact 1 and Fact 2. Taking the bracket with $\ad_V$ we again obtain an element from $\mathcal{D}$, so $X \wedge Y$ is equivalent modulo $\mathcal{D}$ to a linear combination of the terms of the form $H_- \wedge X'$, where $X'=[V, Z_-] \in \m_{2\omega_i} \oplus \m_{2\omega_j}$, and $X_a \wedge Y_a$, where $X_a=[V,H_+], Y_a=Z_+ \in \m_{\omega_i+\omega_j}$ or $X_a=[V,X], Y_a=Z \in \m_{\omega_j}$. Repeatedly using this argument, for every term $X \wedge Y, \; X,Y \in \m_{\omega_i}$, from $K$, we obtain that $K \equiv K_1 \mod(\mathcal{D})$, where $K_1 \in \mathcal{M}$ is a linear combination of the terms $X \wedge Y$ such that either $X \in \ag, \; Y \in \m_{2\omega_j}$ (including $j = i$), or $X, Y \in \m_\a, \; \a \in \Delta^+ \setminus \{\omega_i\}$. Next, suppose that $K$ contains a term $X \wedge Y, \; X,Y \in \m_{2\omega_i}$. Choose $j \ne i$ and take $Z \in \m_{\omega_i+\omega_j}, \; V \in \h_{\omega_i-\omega_j}$. Then $[X,Z] \in \h_{\omega_i-\omega_j}$, so there exist $Z_- \in \m_{\omega_i-\omega_j}$ such that $[X,Z]+[(\omega_i-\omega_j)^*,Z_-]=0$, that is, $X \wedge Z + (\omega_i-\omega_j)^* \wedge Z_-$ belongs to $\mathcal{M}$, hence to $\mathcal{D}$, by Fact 1 and Fact 2. Taking the bracket with $\ad_V$ we again obtain an element from $\mathcal{D}$, so $X \wedge [V,Z]$ is equivalent modulo $\mathcal{D}$ to a linear combination of the terms $[X,V] \wedge Z$, with $[X,V], Z \in \m_{\omega_i+\omega_j}$ and $\theta_{\omega_i-\omega_j}V \wedge Z_-$, with $\theta_{\omega_i-\omega_j}V, Z_- \in \m_{\omega_i-\omega_j}$. It follows that $X \wedge [\h_{\omega_i-\omega_j}, \m_{\omega_i+\omega_j}] \subset (\m_{\omega_i+\omega_j} \wedge \m_{\omega_i+\omega_j}) \oplus (\m_{\omega_i-\omega_j} \wedge \m_{\omega_i-\omega_j}) \oplus \mathcal{D}$. But
$[\h_{\omega_i-\omega_j}, \m_{\omega_i+\omega_j}] \subset \m_{2\omega_i}+\m_{2\omega_j}$, and $\pi_{\m_{2\omega_i}}[\h_{\omega_i-\omega_j}, \m_{\omega_i+\omega_j}]=\m_{2\omega_i}$ (otherwise, there were a nonzero $W \in \m_{2\omega_i}$ such that $[W,\m_{\omega_i+\omega_j}] \perp \h_{\omega_i-\omega_j}$, which contradicts Lemma~\ref{l:rr}\eqref{it:rr1}). It follows that there exists $Y' \in \m_{2\omega_j}$ such that $X \wedge (Y+Y') \in (\m_{\omega_i+\omega_j} \wedge \m_{\omega_i+\omega_j}) \oplus (\m_{\omega_i-\omega_j} \wedge \m_{\omega_i-\omega_j}) \oplus \mathcal{D}$. But $X \wedge Y' \in \mathcal{D}$, as $[X,Y']=0$, so $X \wedge Y$ is equivalent modulo $\mathcal{D}$ to an element of $(\m_{\omega_i+\omega_j} \wedge \m_{\omega_i+\omega_j}) \oplus (\m_{\omega_i-\omega_j} \wedge \m_{\omega_i-\omega_j})$. Repeatedly using this argument, for every term $X \wedge Y, \; X,Y \in \m_{2\omega_i}$, from $K$, we obtain that $K \equiv K_2 \mod(\mathcal{D})$, where $K_2 \in \mathcal{M}$ is a linear combination of the terms $X \wedge Y$ such that either $X \in \ag, \; Y \in \m_{2\omega_j}$ (including $j = i$), or $X, Y \in \m_\a, \; \a \in \Delta^+ \setminus \{\omega_i, 2\omega_i\}$. But the only terms $X \wedge Y$ in $K_2$ such that $\db(X \wedge Y) \in \h_{2\omega_i}$, are the terms with $X \in \ag, \; Y \in \m_{2\omega_i}$. As $K \in \mathcal{M}$, the sum of all these terms appearing in $K$ also belongs to $\mathcal{M}$, hence to $\mathcal{D}$, by Fact 2. So $K \equiv K' \mod(\mathcal{D})$, where $K' \in \mathcal{M}$ is a linear combination of the terms $X \wedge Y$ such that either $X \in \ag, \; Y \in \m_{2\omega_j}$ (with $j \ne i$), or $X, Y \in \m_\a, \; \a \in \Delta^+ \setminus \{\omega_i, 2\omega_i\}$, as required.
\end{proof}

Note that for complex symmetric spaces, assertion~\ref{it:A[X,Y]} of Lemma~\ref{l:rkge3} follows from assertion~\ref{it:phi0} by \cite[Proposition~4.3]{Pan}. It is not however immediately clear how to carry over this result to the real case, as the commuting variety in the complex case can be reducible and can be strictly bigger than the (Zariski or Euclidean) closure of $\ag \times \ag$ \cite{PY} (the simplest example is the complex projective space).

\begin{replemma}{l:Amh}
Let $M_0$ be an irreducible compact symmetric space. Suppose that either $\rk M_0 \ge 2$ or $M_0=\Oc P^2$. Let a linear map $A: \m \to \h$ satisfy
\begin{equation}\label{eq:Amhdef}
    \sigma_{XYZ}\<[X, Y],AZ\>=0,
\end{equation}
for all $X, Y, Z \in \m$. Then there exists $T \in \m$ such that $A=\ad_T$.
\end{replemma}

\begin{proof}
Clearly, the $\h$-submodule of those $A \in \Hom(\m, \h)$ which satisfy \eqref{eq:Amhdef} contains the submodule $\ad_\m$, by the Jacobi identity. We want to show that they coincide.

First consider the case when $\rk M_0 \ge 2$. Let $\ag \subset \m$ be a Cartan subspace.

\smallskip

\emph{Step 1. For any Cartan subspace $\ag \subset \m$ there exists $T' \in \m$ such that for all $\a \in \Delta$, the operator $A'=A+\ad_{T'}:\m \to \h$ satisfies}
\begin{equation}\label{eq:AT'} 
    A'\ag \subset \h_0, \qquad A'\m_\a \subset \bigoplus\nolimits_{\b \in \Delta, \b \parallel \a}\h_\b.
\end{equation}

Taking $X, Y \in \ag, \; Z \in \m_\a$ in \eqref{eq:Amhdef} we get $\a(Y)\<\theta_\a Z, AX\>=\a(X)\<\theta_\a Z, AY\>$, so for any $\a \in \Delta^+$,
there exists $U_\a \in \h_\a$ such that for all $X \in \ag, \quad \pi_{\h_\a}AX=\a(X) U_\a$. Define $T'=\sum_{\a \in \Delta^+} \theta_\a^{-1}U_\a \in \m$. Then for all $X \in \ag, \quad (A+\ad_{T'})X \subset \h_0$. This proves the first formula of \eqref{eq:AT'}. Note that the map $A'=A+\ad_{T'}$ still satisfies \eqref{eq:Amhdef}.

Taking now $X \in \ag, \; Y \in \m_\a, \; Z \in \m_\b$ in \eqref{eq:Amhdef}, with $A$ replaced by $A'$, we obtain $\a(X)\<\theta_\a Y, A'Z\>= \b(X)\<\theta_\b Z, A'Y\>$. If $\b \not\parallel \a$, we can choose $X \in \ag$ such that $\a(X)=0 \ne \b(X)$. Then $A'\m_\a \subset \oplus_{\b \in \Delta, \b \parallel \a}\h_\b \oplus \h_0$, for all $\a \in \Delta$. To prove the second inclusion of \eqref{eq:AT'}, we need to show that there is no $\h_0$-component on the right-hand side.

This is trivially true, if $\rk M_0=\rk \g$, as then $\h_0 =0$. Otherwise, suppose that for some $Z \in \m_\a$, the vector $U=\pi_{\h_0} A'Z$ is
nonzero. Then taking $X, Y \in \m_\b, \; \a \ne \pm \b, \pm 2\b$, in \eqref{eq:Amhdef}, with $A$ replaced by $A'$, we get $\<[X,Y],U\>=0$, so $U \perp [\m_\b,\m_\b]$. As $U \in \h_0$ and as $\m_\b$ is an $\h_0$-module, we have $[U, \m_\b]=0$, hence $[U, \h_\b]=0$. Let $\gamma$ be one of the shortest roots proportional to $\a$, so that $\gamma=\pm \a$ or $\gamma = \pm\frac12 \a$. Then for all $\b \nparallel \a$, we have $\pm \gamma \nparallel \b$, so $0=[U,[\m_\b,\m_\gamma]]=[\m_\b,[U,\m_\gamma]]$.

First suppose that $[U,\m_\gamma]=0$. If $2\gamma \notin \Delta$, then $[U, \m_\b]=0$, for all $\b \nparallel \a$ and for all $\b \parallel \a$. As $U \in \h_0$, we also have $[U, \ag]=0$, so $[U, \m]=0$, a contradiction. If $2\gamma \in \Delta$, then $[U,\h_\gamma]=0$, hence $[U, [\m_\gamma, \m_\gamma] +[\h_\gamma,\h_\gamma]]=0$. But $[\m_\gamma,\m_\gamma]+[\h_\gamma,\h_\gamma] \supset \h_{2\gamma}$ by \cite[Section~3.2]{NT}, so $[U, \h_{2\gamma}]=0$, hence $[U, \m_{2\gamma}]=0$. This again implies $[U, \m]=0$, a contradiction.

Suppose now that $[U,\m_\gamma] \ne 0$. Let $X \in [U,\m_\gamma] \subset \m_\gamma$ be nonzero. We have $[X, \m_\b]=0$, for all $\b \nparallel \a$, so by Lemma~\ref{l:rr}\eqref{it:rr1}, every root not proportional to $\a$ is orthogonal to $\a$, hence $\Delta$ is a union of two nonempty orthogonal subsets, which contradicts the fact that $M_0$ is irreducible.

This proves the second formula of \eqref{eq:AT'}.

\smallskip

\emph{Step 2. For a Cartan subspace $\ag \subset \m$, define $T' \in \m$ and $A'=A+\ad_{T'}:\m \to \h$, as in Step 1. Then for any $\a \in \Delta$
such that $\br\a \cap \Delta = \pm \a$, we have $A'\a^*=0$ \emph{(}where $\a^* \in \ag$ is dual to $\a$\emph{)} and there exists
$c_\a \in \br$ such that $A'X=c_\a \theta_\a X$, for all $X \in \m_\a$}.

Denote $\tilde \m_a =\m_\a \oplus \br \a^*$. Then by \cite[Lemma~2.25]{Nag}, $\tilde \m_a$ is a Lie triple system tangent to a totally geodesic submanifold of constant positive curvature and moreover, for $X, Y \in \tilde \m_\a$, the map $\iota: X \wedge Y \mapsto \|\a\|^{-2}[X,Y]$ is an isomorphism of Lie algebras $\so(\tilde \m_\a)$ and $[\tilde \m_\a,\tilde \m_\a]= \h_\a \oplus [\m_\a,\m_\a]$ (note that the latter subspace lies in $\h_0$), and for $U \in [\tilde \m_\a,\tilde \m_\a], \; X \in \tilde \m_\a$, we have $[U, X] = \iota^{-1}U (X)$. 

For any nonzero $X \in \tilde \m_a$, the subspace $\ag_X=\Ker \a \oplus \br X$ is a Cartan subspace, with $\m_{\a,X}=\tilde \m_\a \cap X^\perp$ and
$\h_{\a,X}=[X,\m_{\a,X}]$ the root spaces. Then by Step 1 applied to $\ag_X$ and $A'$, there exists $T'(X) \in \m$, with $T'(\a^*)=0$, such that
the map $A'+\ad_{T'(X)}$ satisfies \eqref{eq:AT'}, that is, $[(A'+\ad_{T'(X)})\ag_X,\ag_X]=0$ and
$(A'+\ad_{T'(X)})\m_{\a,X} \subset \h_{\a,X}=[X,\m_{\a,X}]$. From the first equation, it follows that $[(A'+\ad_{T'(X)})\Ker \a,\Ker \a]=0$, for all
$X \in \tilde \m_a$, which (as $T'(\a^*)=0$) implies $[[T'(X),\Ker \a],\Ker \a]=0$. As the only roots proportional to $\a$ are $\pm \a$, this implies
$T'(X) \subset \Ker \a \oplus \tilde \m_\a$. Then from the above,
$(A'+\ad_{T'(X)})\m_{\a,X} \subset \h_{\a,X}=[X,\m_{\a,X}]=[X,\tilde \m_\a]$, so for any $Y \in \tilde \m_\a, \; Y \perp X$, we get
$A'Y+[T'(X),Y] \in [X,\tilde \m_\a]$, so $A'Y \in [X,\tilde \m_\a]-[T'(X),Y]=[X,\tilde \m_\a]-[S(X),Y]$, where $S(X)$ is the $\m_{\a, X}$-component
of $T'(X)$ (so $S(X) \in \tilde \m_\a, \; S(X) \perp X$). It follows that $A'Y \in [\tilde \m_\a,\tilde \m_\a]$, which is
isomorphic to the Lie algebra $\so(\tilde \m_\a)$ via the isomorphism $\iota: X \wedge Y \mapsto \|\a\|^{-2}[X,Y]$. Therefore for all
$X, Y \in \tilde \m_\a, \;  X \perp Y, \; X \ne 0$,
\begin{equation}\label{eq:iota}
\iota^{-1}A'Y = \|\a\|^2(X \wedge F(X,Y) - S(X) \wedge Y),
\end{equation}
where $S(X), F(X,Y) \in \tilde \m_\a, \; S(X) \perp X$ and $S(\a^*)=0$ (as $T'(\a^*)=0$). Moreover, from the fact that
$[(A'+\ad_{T'(X)})\ag_X,\ag_X]=0$ we obtain that $[A'X+[T'(X),X],X]=0$, for all nonzero $X \in \tilde \m_\a$. As
$[T'(X),X]=[S(X),X]$, we get $\iota^{-1}(A'X+[S(X),X])X=0$, so for all $X \in \tilde \m_\a$,
\begin{equation}\label{eq:iotaxx}
\iota^{-1}A'X(X)=-\|\a\|^2 (S(X) \wedge X)X= \|\a\|^2 \| X\|^2 S(X),
\end{equation}
by the definition of $\iota$ and from the fact that $S(X) \perp X$.

Now, if $m_\a (= \dim \m_\a)=1$, then the second statement of Step 2 follows trivially. To prove the first statement, take $Y = \a^*$ in
\eqref{eq:iota}. As $\dim \tilde \m_\a =2$, we get $\iota^{-1}A'\a^* = c X \wedge \a^*$, for some $c \in \br$, where $X$ spans $\m_\a$. But then
by \eqref{eq:iotaxx}, $-c\|\a\|^2 X=\iota^{-1}A'\a^*(\a^*) = 0$, as $S(\a^*)=0$, so $c = 0$, that is, $A'\a^* = 0$.

If $m_\a > 2$, then \eqref{eq:iota} implies $\<\iota^{-1}A'Y (Z_1) , Z_2\>=0$, for any $Z_1, Z_2 \in \tilde \m_\a, \; Z_1,Z_2 \perp X, Y$, hence
for any $Z_1, Z_2 \in (\tilde \m_\a \cap Y^\perp)$. Taking $Z \in \tilde \m_\a, \; Z \perp X, Y$, we obtain from \eqref{eq:iota} that
$0=\<\iota^{-1}A'Y (X),Z\> = \|\a\|^2 \|X\|^2 \<F(X,Y), Z\>$, so $F(X, Y) \in \Span(X, Y)$. It now follows from \eqref{eq:iota} that for all
$Y \in \tilde \m_\a$, $\iota^{-1}A'Y \in \tilde \m_\a \wedge Y$, so by linearity, there exists $P \in \tilde \m_a$ such that
$\iota^{-1}A'Y = P \wedge Y$. But then from the fact that $S(\a^*)=0$ we obtain by \eqref{eq:iotaxx} that
$0=\iota^{-1}A'\a^*(\a^*) = (P \wedge \a^*)\a^*$, so $P= c \a^*$, for some $c \in \br$. Then $A'\a^* = 0$ and
$A'Y=\iota(c\a^* \wedge Y)=c\|\a\|^{-2}[\a^*,Y]=c\|\a\|^{-2} \theta_\a Y$, for all $Y \in \m_\a$, as required.

Finally, if $m_\a=2$, then $\dim \tilde \m_\a = 3$, so the Lie algebra $\so(\tilde\m_\a)$ is isomorphic to $\tilde \m_\a$ with the cross-product,
with the isomorphism $v$ defined by $v(X_1 \wedge X_2)=X_1 \times X_2$. Acting by $v$ on the both sides of \eqref{eq:iota} and introducing
$w \in \End(\tilde \m_\a)$ by $wY=v(\iota^{-1}A'Y)$ we get $wY = \|\a\|^2(X \times F(X,Y) - S(X) \times Y)$, for all $X \perp Y, \; X \ne 0$, so
$\<wY, X\> = \|\a\|^2 \<S(X) \times X,  Y\>$, which implies $w^tX= \|\a\|^2 S(X) \times X+f(X)X$, where $f:\tilde \m_\a \to \br$ and $w^t$ is the
operator adjoint to $w$. Then $S(X)=\|\a\|^{-2}\|X\|^{-2} w^tX \times X$, as $S(X) \perp X$. On the other hand, from \eqref{eq:iotaxx} we obtain
$\|\a\|^2 \| X\|^2 S(X) = \iota^{-1}A'X(X)= wX \times X$, so $(w^t+w)X \times X = 0$. It follows that $w^t+w= 2 c \, \id$, for some $c \in \br$,
so $wX=cX+P \times X$ for some $P \in \tilde \m_\a$. Then $S(X)=\|\a\|^{-2}\|X\|^{-2} w^tX \times X=-\|\a\|^{-2}\|X\|^{-2} (P \wedge X) X$. As
$S(\a^*)=0$, we get $P=c_1 \a^*$ for some $c_1 \in \br$. Then $v(\iota^{-1}A'X)=wX=cX+c_1 \a^* \times X$, so, by the definition of $v$ and of $\iota$, $[A'X,Y]=(\iota^{-1}A'X)(Y)=cX \times Y +c_1 (\a^* \wedge X)Y$, for any $X, Y \in \tilde \m_\a$. Then $\<[A'X,Y], Z\>= c(X,Y,Z) +c_1 (\<\a^*, Y\> \<X,Z\>-\<\a^*, Z\> \<X,Y\>)$, for all $X, Y, Z \in \tilde \m_\a$, where $(X,Y,Z)$ is the triple product in the three-dimensional Euclidean space $\tilde \m_\a$. Taking the cyclic sum by orthonormal $X, Y, Z \in \tilde \m_\a$ and using the fact that $A'$ satisfies \eqref{eq:Amhdef} we get $c=0$. Then $\iota^{-1}A'X=c_1 \a^* \wedge X$, so $A'\a^*=0$ and $A'X = c_1 \theta_\a X$, for $X \in \m_\a$, as required.

\smallskip

\emph{Step 3. For a Cartan subspace $\ag \subset \m$, define $T' \in \m$ and $A'=A+\ad_{T'}:\m \to \h$, as in Step 1. Then $A'\ag=0$ and there exists
a linear form $c$ on $\ag^*$ such that for all $\a \in \Delta, \quad A'X=c(\a) \theta_\a X$, for all $X \in \m_\a$}.

First suppose that the root system $\Delta$ is reduced. Then the first statement of Step 3 immediately follows from Step 2. Also, from Step 2 we know that for every $\a \in \Delta$, there exists a constant $c_\a$ such that $A'_{|\m_\a}=c_\a \theta_\a \id_{|\m_\a}$. It remains to show that the function $\a \mapsto c_\a$ is a restriction of a linear form on $\ag^*$ to $\Delta$. Choose a subsystem $\Delta^+$ of positive roots and let $\a_1, \dots, \a_r \in \Delta^+, \; r = \rk M_0$, be a basis of simple roots. Then for every $\b \in \Delta^+$ we have $\beta=\sum_{i=1}^r n_i \a_i$, with all $n_i$ being nonnegative integers. We will show that for every $\b \in \Delta^+, \quad c_\beta=\sum_{i=1}^r n_i c_{\a_i}$ by induction by $h(\b)=\sum_{i=1}^r n_i$, the height of $\b$. For the roots of height one (for simple roots), this is trivial. Suppose that for all the roots of height less than $h_0 \ge 2$ the above equation holds. Let $h(\b)=h_0$. Then $\<\b, \a\> > 0$ for some simple root $\a =\a_i$ (otherwise $\<\b,\b\> \le 0$), so $\b=\gamma+\a$ for some $\gamma \in \Delta^+$ (note that $h(\gamma)=h_0-1$) and the $\h_\gamma$-component of $[\m_\b, \m_\a]$ is nonzero by Lemma~\ref{l:rr}\eqref{it:rr2}. Then we can choose $X_\a \in \m_\a, \; X_\b \in \m_\b$ and $X_\gamma \in \m_\gamma$ in such a way that $\<[X_\a, X_\b], \theta_\gamma X_\gamma\> \ne 0$. By Lemma~\ref{l:rr}\eqref{it:rr3}, $-\<[X_\gamma, X_\a], \theta_\b X_\b\>=\<[X_\b, X_\gamma], \theta_\a X_\a\> =\<[X_\a, X_\b], \theta_\gamma X_\gamma\>$. Substituting such $X_\a, X_\b, X_\gamma$ into \eqref{eq:Amhdef}, with $A$ replaced by $A'$, we get $\<[X_\a, X_\b], \theta_\gamma X_\gamma\>(c_\gamma+c_\a-c_\b)=0$, so $c_\b=c_\gamma+c_\a$, as required. The fact that $c_{-\a}=-c_\a$ now follows from the fact that $\theta_{-\a}=-\theta_\a$. This proves the second statement of Step 3 for a reduced system $\Delta$.

Now consider the case of a non-reduced root system. Every such system is of type $\mathrm{BC}_r$, so that $\Delta=\{\pm \omega_i, \pm 2\omega_i, \pm \omega_i \pm \omega_j\},\; 1 \le i < j \le r$. The first statement of Step 3 now follows by linearity from the first statement of Step 2 applied to the roots $\pm \omega_i \pm \omega_j, \; i<j$. From the second statement of Step 2 we also obtain that for all $\a=\pm\omega_i\pm\omega_j, \; i \ne j$, there exists $c_\a \in \br$ with $A'X = c_\a \theta_\a X$, for all $X \in \m_\a$ (note that $c_{-\a}=-c_\a$, as $\theta_{-\a}=-\theta_\a$). Substituting $X \in \m_{\omega_i+\omega_j},\; Y \in \m_{\omega_i-\omega_j},\; Z \in \m_{\omega_i}, \; i \ne j$, into \eqref{eq:Amhdef}, with $A$ replaced by $A'$, we get $\<[X,Y], A'Z\> =0$, which implies $\<[X,Y], \pi_{2\omega_i}A'Z\> =0$ by \eqref{eq:AT'}. Let $V=\theta_{2\omega_i}^{-1}\pi_{2\omega_i}A'Z \in \m_{2\omega_i}$. Then by Lemma~\ref{l:rr}\eqref{it:rr3} we obtain $\<[V,X], \theta_{\omega_i-\omega_j}Y\> =0$, which implies that $V=0$ by Lemma~\ref{l:rr}\eqref{it:rr1}. It follows that $\pi_{2\omega_i} A'\m_{\omega_i} =0$, so $A'\m_{\omega_i} \subset \h_{\omega_i}$ by \eqref{eq:AT'}. Taking now $X \in \ag, \; Y \in \m_{\omega_i}, \; Z \in \m_{2\omega_i}$ in \eqref{eq:Amhdef}, with $A$ replaced by $A'$, we obtain $\<\theta_{\omega_i} Y, A'Z\>=0$, so $A'\m_{2\omega_i} \subset \h_{2\omega_i}$ by \eqref{eq:AT'}. Substituting $X \in \m_{\omega_i+\omega_j},\; Y \in \m_{\omega_i-\omega_j},\; Z \in \m_{2\omega_i}, \; i \ne j$, into \eqref{eq:Amhdef}, with $A$ replaced by $A'$, and using Lemma~\ref{l:rr}\eqref{it:rr3} we get $\<[X,Y], (A'-(c_{\omega_i+\omega_j}+c_{\omega_i-\omega_j})\theta_{2\omega_i})Z\> =0$, so $\<[(A'-(c_{\omega_i+\omega_j}+c_{\omega_i-\omega_j})\theta_{2\omega_i})Z,X], Y\>=0$. As we already know that
$(A'-(c_{\omega_i+\omega_j}+c_{\omega_i-\omega_j})\theta_{2\omega_i})Z \in \h_{2\omega_i}$ and as $2\omega_i+(\omega_i+\omega_j)$
is not a root, it follows that $0=[(A'-(c_{\omega_i+\omega_j}+c_{\omega_i-\omega_j})\theta_{2\omega_i})Z,X]
=[\theta_{2\omega_i}(\theta_{2\omega_i}^{-1} A'-(c_{\omega_i+\omega_j}+c_{\omega_i-\omega_j})\id)Z,X]$, for all
$X \in \m_{\omega_i+\omega_j},\; Z \in \m_{2\omega_i}, \; i \ne j$. By Lemma~\ref{l:rr}\eqref{it:rr1} we obtain
$(\theta_{2\omega_i}^{-1} A'-(c_{\omega_i+\omega_j}+c_{\omega_i-\omega_j})\id)Z=0$, so
$A'Z=(c_{\omega_i+\omega_j}+c_{\omega_i-\omega_j})\theta_{2\omega_i}Z$, for all $Z \in \m_{2\omega_i}$. It follows that there exist $c_i \in \br$
such that $A'Z=2 c_i \theta_{2\omega_i}Z$, for all $Z \in \m_{2\omega_i}$, and $A'Z=(\ve_1 c_i+\ve_2 c_j)\theta_{\ve_1 \omega_i+\ve_2 \omega_j}Z$,
for all $Z \in \m_{\ve_1 \omega_i+\ve_2 \omega_j}, \; \ve_1, \ve_2 = \pm 1$. To finish the proof, it remains to show that for all $i=1, \dots, r$,
we have $A'X=c_i \theta_{\omega_i}X$, for all $X \in \m_{\omega_i}$.

Substituting $X, Y \in \m_{\omega_i}, \; Z \in \ag$ into \eqref{eq:Amhdef}, with $A$ replaced by $A'$, and using the first statement of this step we get $\<\theta_{\omega_i}X, A'Y\>=\<\theta_{\omega_i}Y, A'X\>$, so there exist symmetric endomorphisms $S'_i \in \Sym(\m_{\omega_i})$ such that $A'X=\theta_{\omega_i} S'_iX$, for all $X \in \m_{\omega_i}$. Denote $S_i=S_i'-c_i \id$. Substituting $X, Y \in \m_{\omega_i}, \; Z \in \m_{2\omega_i}$ into \eqref{eq:Amhdef}, with $A$ replaced by $A'$, we obtain $2c_i\<[X,Y],\theta_{2\omega_i}Z\>+\<[Y,Z],\theta_{\omega_i}S'_iX\>+\<[Z,X],\theta_{\omega_i}S'_iY\>=0$, so
$c_i\<[X,Y],[\omega_i^*,Z]\>+\<[Y,Z],[\omega_i^*,S'_iX]\>+\<[Z,X],[\omega_i^*,S'_iY]\>=0$, which by the Jacobi identity implies
$\<[Y,Z],[\omega_i^*,S_iX]\>+\<[Z,X],[\omega_i^*,S_iY]\>=0$, that is, the operator
$(\ad_Z\ad_{\omega_i^*})_{|\m_{\omega_i}}S_i \in \End(\m_{\omega_i})$ is symmetric. But the operator
$(\ad_Z\ad_{\omega_i^*})_{|\m_{\omega_i}} \in \End(\m_{\omega_i})$ is skew-symmetric. Indeed, for any $X \in \m_{\omega_i}$ we have
$[[[\omega_i^*,X],X],\omega_i^*]=-[[[X,\omega_i^*],\omega_i^*],X]=0$, so $[[\omega_i^*,X],X] \parallel \omega_i^*$, as
$\br\omega_i^* \oplus \m_{\omega_i} \oplus \m_{2\omega_i}$ is a Lie triple system tangent to a rank one symmetric space. So
$\<[[\omega_i^*,X],X], Z\>=0$, for any $Z \in \m_{2\omega_i}$, which implies that $(\ad_Z\ad_{\omega_i^*})_{|\m_{\omega_i}} \in \End(\m_{\omega_i})$
is skew-symmetric, hence $(\ad_{[Z,\omega_i^*]})_{|\m_{\omega_i}}=2(\ad_Z\ad_{\omega_i^*})_{|\m_{\omega_i}}$. It follows that for all $U \in \h_{2\omega_i}$, the operator $\ad_{U|\m_{\omega_i}}S_i \in \End(\m_{\omega_i})$ is symmetric, so $\ad_{U|\m_{\omega_i}}S_i = -S_i\ad_{U|\m_{\omega_i}}$. Therefore, for every eigenvalue $\lambda$ of $S_i$, with the corresponding eigenspace $E(\la) \subset \m_{\omega_i}$, $-\lambda$ is also an eigenvalue and moreover, $[U,E(\la)] =E(-\la)$, for any nonzero $U \in \h_{2\omega_i}$ (note that the restriction of $\ad_U$ to $\m_{\omega_i}$ is onto). Now, the dimension $m_{2\omega_i}$ can be only $1, 3$ or $7$. 
In the latter case, $M_0$ is the Cayley projective plane, which is of rank one. If $m_{2\omega_i}=3$, the action of $\h_{2\omega_i}$ defines a
quaternionic structure on $\m_{\omega_i}$, so, with an appropriate choice of $U_1,U_2,U_3 \in \h_{2\omega_i}$, the restriction of $\ad_{U_1}\ad_{U_2}\ad_{U_3}$ to $\m_{\omega_i}$ is the identity. As each of them permutes the eigenspaces $E(\la)$ and $E(-\la)$, we get $S_i=0$. Consider the case $m_{2\omega_i}=1$ (then the space $M_0$ is Hermitian).

Substituting $X \in \m_{\omega_i}, \; Y \in \m_{\omega_j}, Z \in \m_{\omega_i+\omega_j},\; i \ne j$, into \eqref{eq:Amhdef}, with $A$ replaced by $A'$, and using Lemma~\ref{l:rr}\eqref{it:rr3} we obtain $\<[S_iX,Y]+[X,S_jY], \theta_{\omega_i+\omega_j}Z\>=0$, that is, $[S_iX,Y]+[X,S_jY] \in \h_{\omega_i-\omega_j}$. Similarly, taking $Z \in \m_{\omega_i-\omega_j}$ we get $[S_iX,Y]-[X,S_jY] \in \h_{\omega_i+\omega_j}$. It follows that for the eigenspaces $E(\la_a) \subset \m_{\omega_i}, \; E(\mu_b) \in \m_{\omega_j}$ of the operators $S_i, S_j$, respectively, with the corresponding eigenvalues $\la_a, \mu_b$, we have
\begin{equation}\label{eq:EE}
(\la_a+\mu_b)[E(\la_a),E(\mu_b)] \subset \h_{\omega_i-\omega_j}, \quad (\la_a-\mu_b)[E(\la_a),E(\mu_b)] \subset \h_{\omega_i+\omega_j}.
\end{equation}
Suppose $\la_a \ne 0$ and let $X \in E(\la_a)$ be nonzero. Then for all $Y \in E_a:=\oplus_{\mu_b \ne-\la_a} E(\mu_b) \subset \m_{\omega_j}$, we have
$[X, Y] \in \h_{\omega_i-\omega_j}$, so $\<[X, E_a], \h_{\omega_i+\omega_j}\>=0$, that is, $\ad_X \h_{\omega_i+\omega_j} \perp E_a$. As the map
$\ad_X: \h_{\omega_i+\omega_j} \to \m_{\omega_j}$ is surjective by Lemma~\ref{l:rr}\eqref{it:rr1}, we obtain
$m_{\omega_i+\omega_j}+\dim E_a \le m_{\omega_j}$. But $2\dim E_a \ge m_{\omega_j}$ (as it is shown in the previous paragraph, for every eigenspace
$E(\mu_b) \subset \m_{\omega_j}, \; \mu_b \ne 0$ of $S_j$, there is an eigenspace $E(-\mu_b) \subset \m_{\omega_j}$ of the same dimension). Then
$2m_{\omega_i+\omega_j} \le m_{\omega_j}$. Inspecting the multiplicities of the restricted roots from the Satake diagrams 
we obtain that each of $S_i$ is zero in all the cases, except possibly, for the complex Grassmannian
$M_0=SU(p+q)/S(U(p)\times U(q)),\; p>q>1$. In the latter case, an easy direct computation of the Lie brackets shows that
for $X \in \m_{\omega_i}, \; Y \in \m_{\omega_j}$, we have $[X,Y] \subset \h_{\omega_i-\omega_j} \cup \h_{\omega_i+\omega_j} \iff [X,Y]=0$. Then
from \eqref{eq:EE}, $[E(\la_a),E(\mu_b)]=0$, unless $\la_a=\mu_b=0$. Therefore, if $\la_a \ne 0$ and $X \in E(\la_a)$ is nonzero, we get
$[X, \m_{\omega_j}]=0$, so $\<[X, \m_{\omega_j}], \h_{\omega_i+\omega_j}\>=0$, which implies $\ad_X\h_{\omega_i+\omega_j}=0$, a contradiction
with Lemma~\ref{l:rr}\eqref{it:rr1}. It follows that all the operators $S_i$ vanish.

Thus in all the cases $S_i=0$, so, from the definition of the $S_i$'s we get $A'X=c_i \theta_{\omega_i}X$, for all $X \in \m_{\omega_i}$, as required.

\smallskip

The claim of the lemma now follows, as by Step 3, there exists $c \in \ag$ such that $A'=\ad_c$, hence $A=\ad_{c-T'}$.

\medskip

Now consider the case $M_0=\Oc P^2$. For $X, Y \in \m$, we have
\begin{equation}\label{eq:Rop2}
\begin{gathered}
\ad_{[X,Y]}=3 X \wedge Y + \sum\nolimits_{i=0}^8 (S_i X) \wedge (S_i Y) = 3 X \wedge Y + \sum\nolimits_{i=0}^8 S_i (X \wedge Y) S_i, \\
\text{where } S_i^* = S_i, \text{ and } S_iS_j+S_jS_i=2 \K_{ij} \id, \text{ for all } 0 \le i, j \le 8
\end{gathered}
\end{equation}
(see \cite{Fr1} or \cite[Section~2.3]{Nampa}, where the operators $S_i$ are given explicitly). The operators $S_iS_j, S_iS_jS_k$, $i <j < k$, are skew-symmetric and form a basis for $\so(16)$ (which is orthonormal, if we replace every $S_i$ by $\frac14 S_i$). The isotropy representation of $\h=\so(9)=\Lambda^2 \br^9$ is the spin representation defined by $u_i \wedge u_j \mapsto S_iS_j$, where $u_i, \; 0 \le i \le 8$, is an orthonormal basis for $\br^9$. The irreducible decomposition of the $\h$-module $\so(16)$ is given by $\so(16)=\Lambda^2 \br^9 \oplus \Lambda^3 \br^9$, where $\Lambda^3 \br^9=\Span_{i < j < k}(S_iS_jS_k)$. This decomposition is orthogonal and moreover, by \eqref{eq:Rop2}, $\ad_{[X,Y]}=8\pi_2(X \wedge Y)=\sum_{i<j} \<S_iS_jX, Y\> S_iS_j$, where $\pi_2$ is the orthogonal projection to the submodule $\Lambda^2 \br^9 \subset \so(16)$.

By the assumption, a linear map $A: \m \to \h$ satisfies $\sigma_{XYZ}\<[X, Y],AZ\>=0$, for all $X, Y, Z \in \m$, so $0=\sigma_{XYZ}\<\pi_2(X \wedge Y), AZ\>=\sigma_{XYZ}\<X \wedge Y, \pi_2(AZ)\>=\sigma_{XYZ}\<X \wedge Y, AZ\>$. Here $A$ can be viewed as an element of the $\h$-module $\Lambda^2 \br^9 \otimes \m$, and then the assumption of the lemma means that $A \in \Ker \Xi$, where $\Xi: \Lambda^2 \br^9 \otimes \m \to \Lambda^3\m$ is the homomorphism of $\h$-modules defined by $\Xi((u_i \wedge u_j)\otimes a)=a \wedge (S_iS_j)$, for $\a \in \m, \; u_i, u_j \in \br^9$. The irreducible decomposition of the both modules are known (\cite[Section~7]{Fr1} and \cite{Slu}). Define the $\h$-homomorphisms $\Theta_k: \Lambda^k \br^9 \otimes \m \to \Lambda^{k+1} \br^9 \otimes \m, \; \Theta_k^*: \Lambda^k \br^9 \otimes \m \to \Lambda^{k-1} \br^9 \otimes \m $ by
\begin{equation}\label{eq:modules}
\Theta_k (\omega\otimes a)=\sum\nolimits_{i=0}^8 (u_i \wedge \omega) \otimes S_ia, \qquad \Theta_k^* (\omega\otimes a)=-\sum\nolimits_{i=0}^8 (u_i \lrcorner \omega) \otimes S_ia
\end{equation}
and denote $P_k = \Ker \Theta_k^*$. Then we have irreducible decompositions
\begin{equation}\label{eq:spin}
\Lambda^2 \br^9 \otimes \m =\Theta_1 \Theta_0(P_0) \oplus \Theta_1 (P_1) \oplus P_2, \qquad \Lambda^3\m \simeq P_1 \oplus P_2,
\end{equation}
with $\Theta_1 \Theta_0: P_0 \to \Theta_1 \Theta_0(P_0)$ and $\Theta_1: P_1 \to \Theta_1 (P_1)$ being isomorphisms on their images. Now $P_0=\m$ and for $T \in \m$ we have $\Theta_1 \Theta_0(T)=\sum_{ij}S_iS_jT \otimes (u_i \wedge u_j)$ and so $(\Xi\Theta_1 \Theta_0(T))(X,Y,Z)=\sigma_{XYZ}\sum_{ij} \<S_iS_jT, X\>\<S_iS_jY, Z\>=\sigma_{XYZ}\sum_{ij} \<\ad_{[Y, Z]}T, X\>=0$, by the Jacobi identity. It follows that $\Ker \Xi \supset \Theta_1 \Theta_0(P_0) \simeq \m$, so to prove the lemma it suffices to show that $\Xi$ maps the remaining two irreducible components of $\m \otimes \Lambda^2 \br^9$ from \eqref{eq:spin} onto their images isomorphically, that is, it suffices to produce an element in each of these components which does not belong to $\Ker \Xi$.

We start with $\Theta_1 (P_1)$. By \eqref{eq:modules}, for any $T \in \m$, we have $\Theta_1^*(u_i \otimes T)=-S_iT$, so $u_0 \otimes S_1T+u_1 \otimes S_0T \in P_1 = \Ker(\Theta_1^*)$. Then $\Theta_1 (u_0 \otimes S_1T+u_1 \otimes S_0T)=\sum\nolimits_{i=0}^8 (u_i \wedge u_0) \otimes S_iS_1T+(u_i \wedge u_1) \otimes S_iS_0T$, so $(\Xi\Theta_1 (u_0 \otimes S_1T+u_1 \otimes S_0T))(X,Y,Z)=\sigma_{XYZ}\sum\nolimits_{i=2}^8 (\<S_iS_1T,X\>\<S_iS_0Y,Z\>+\<S_iS_0T,X\>\<S_iS_1Y,Z\>)$. From the commutator relations \eqref{eq:Rop2} it follows that the operators $S_1S_2S_3S_4$ and $S_0$ are symmetric, orthogonal and commuting. Choose $X \in \m$ to be their common eigenvector with eigenvalue $1$, and then choose $Y=S_1S_3X, \; Z=S_2S_3X$ and $T=S_2X$. Using relations \eqref{eq:Rop2} we then obtain that $(\Xi\Theta_1 (u_0 \otimes S_1T+u_1 \otimes S_0T))(X,Y,Z)=-3\|X\|^4$. It follows that the restriction of $\Xi$ to $\Theta_1 (P_1)$ is an isomorphism onto the image.

We next consider $P_2$. By \eqref{eq:modules} we have $\Theta_2^*((u_i \wedge u_j)\otimes T)=u_i \otimes S_jT-u_j \otimes S_iT$, so for any $T \in \m$, the element $N=(u_0 \wedge u_1) \otimes S_1S_0T+(u_1 \wedge u_2) \otimes S_1S_2T+(u_2 \wedge u_3) \otimes S_3S_2T+(u_3 \wedge u_0) \otimes S_3S_0T$ lies in $P_2  = \Ker(\Theta_2^*)$. We have $\Xi(N)(X,Y,Z)=\sigma_{XYZ}(\<S_1S_0T,X\>\<S_0S_1Y,Z\>+\<S_1S_2T,X\>\<S_1S_2Y,Z\>+\<S_3S_2T,X\>\<S_2S_3Y,Z\> +\<S_3S_0T,X\>\<S_3S_0Y,Z\>)$. From the commutator relations \eqref{eq:Rop2} it follows that the operator $S_0S_1S_2S_3$ is symmetric, orthogonal and has zero trace. Choose a nonzero $X \in \m$ to satisfy $\<S_0S_1S_2S_3X,X\>=0$ and then choose $T=S_0S_3X, \; Z = S_0S_3Y$ and a nonzero $Y \in \m$ such that $Y \perp X, S_0S_1S_2S_3X, S_iS_jX$, for all $i,j=0,1,2,3$. Then using the commutator relations \eqref{eq:Rop2} again we obtain that $\Xi(N)(X,Y,Z)=\|X\|^2\|Y\|^2$, hence the restriction of $\Xi$ to $P_2$ is also an isomorphism onto the image.

So $\Ker \Xi= \Theta_1 \Theta_0(P_0) \simeq \m$, as required.
\end{proof}

\begin{remark} Note that in the case when $M_0$ is a quaternionic projective space of dimension $4m < 20$ or a complex projective space, the claim of Lemma~\ref{l:Amh} is false, by the dimension count.
\end{remark}

\section{Symmetric spaces of rank two. Proof of Lemma~\ref{l:rk2}}
\label{s:rk2}

In this section, we give the proof of Lemma~\ref{l:rk2} from Section~\ref{s:irred}:

\begin{replemma}{l:rk2}
Suppose $M_0$ is a compact irreducible symmetric space of rank two other than $SU(3)/SO(3)$. In the assumptions of Proposition~\ref{p:irre}, $\Phi=0$.
\end{replemma}

\begin{proof}
Compact irreducible symmetric spaces of rank two, modulo low-dimensional isomorphisms, are: 
\begin{itemize}
  \item the compact groups $SU(3), \; Sp(2), \; G_2$;
  \item the Grassmannians
  $SO(p+2)/(SO(p) \times SO(2)), \; p \ge 3,\; SU(p+2)/S(U(p) \times U(2)), \; p \ge 3$, $Sp(p+2)/(Sp(p) \times Sp(2)), \; p \ge 2$;
  \item three classical spaces $SO(10)/U(5), \; SU(6)/Sp(3), \; SU(3)/SO(3)$;
  \item three exceptional spaces $E_6/F_4, \; E_6/(SO(10) \times SO(2)), \; G_2/SO(4)$.
\end{itemize}

For the groups, the claim follows from \cite[Proposition]{Ndga}.

Note that in general, it is sufficient to prove that $\<\Phi(Y,X),X\>=0$ for all $X,Y \in \m$, as then the map $(X,Y,Z)\mapsto\<\Phi(X,Y),Z\>$ is skew-symmetric, so is zero by \eqref{eq:cyclePhi}. Now, given an arbitrary $X \in \m$, consider a Cartan subalgebra $\ag \subset \m$ containing $X$. By linearity, it is sufficient to show that $\<\Phi(Y,X),X\>=0$, when $Y$ is either a root vector or belongs to $\ag$. So it suffices to prove that for every $\a \in \Delta$ and every $Y \in \m_\a$, we have $\<\Phi(\ag_Y,\ag_Y),\ag_Y\>=0$, where $\ag_Y=\ag \oplus \br Y$. Suppose $\m' \subset \m$ is an irreducible Lie triple system containing $\ag_Y$. Then by Remark~\ref{r:tg}, he maps $K$ and $\Phi$ on $\m$ descend to the maps $K'$ and $\Phi'$ on $\m'$, which still satisfy the assumptions of Proposition~\ref{p:irre}. As $\<\Phi'(X,Y),Z\>=\<\Phi(X,Y),Z\>$, for all $X,Y,Z \in \m'$, it is sufficient to prove the lemma for some irreducible Lie triple system $\m'$ containing $\ag_Y$.

Now, the spaces $SU(6)/Sp(3)$ and $E_6/F_4$ have the restricted root system of type $\mathrm{A}_2$ and each of them has a totally geodesic submanifold $SU(3)$ of the maximal rank \cite{Kl}. The Lie triple system $\m'$ tangent to $SU(3)$ is again of type $\mathrm{A}_2$; it contains a Cartan subalgebra $\a \subset \m$ and can be rotated by the isotropy subgroup of $\ag$ to contains a given root vector $Y$ of $\m$ (as the Weyl group is transitive on the roots of the equal length, and as all the roots of the system $\mathrm{A}_2$ have the same length). The claim now follows from the fact that $\Phi=0$ for $SU(3)$.

The spaces $SU(p+2)/S(U(p) \times U(2)), \; p \ge 3, \; Sp(p+2)/(Sp(p) \times Sp(2)), \; p \ge 3,\; E_6/(SO(10) \times SO(2))$, and $SO(10)/U(5)$,  have the restricted root system of type $\mathrm{BC}_2$. Each of them contains a totally geodesic submanifold $SU(5)/S(U(3) \times U(2))$ of the maximal rank and with the root system of type $\mathrm{BC}_2$ \cite{Kl}. By the action of the isotropy group, the Lie triple system $\m'$ tangent to $SU(5)/S(U(3) \times U(2))$ can be chosen to contain the given Cartan subalgebra $\ag \subset \m$ and then, as the root system of $\m'$ contains the roots of all lengths, can be rotated by the isotropy subgroup of $\ag$ to contains a given root vector $Y$ of $\m$. Hence to prove the lemma for all these spaces it suffices to show that $\Phi=0$ for $SU(5)/S(U(3) \times U(2))$. We can reduce the space further by noting that if the root vector $Y$ of the space $\ag_Y$ corresponds to the longest or the second longest root of $SU(5)/S(U(3) \times U(2))$, then $\ag_Y$ is contained in a Lie triple system of type $\mathrm{B}_2$ tangent to a totally geodesic $SO(6)/(SO(4) \times SO(2))= SU(4)/S(U(2) \times U(2)) \subset SU(5)/S(U(3) \times U(2))$. If the root vector $Y$ corresponds to the shortest root, then $\ag_Y$ is again contained in a Lie triple system of type $\mathrm{B}_2$ tangent to a totally geodesic $SO(5)/(SO(3) \times SO(2)) \subset SU(5)/S(U(3) \times U(2))$. Hence to prove the lemma for all these spaces it suffices to show that $\Phi=0$ for the Grassmannians $SO(p+2)/(SO(p) \times SO(2)), \; p = 3, 4$, which are included in the next case.

The spaces $Sp(4)/(Sp(2) \times Sp(2))$ and $SO(p+2)/(SO(p) \times SO(2)), \; p \ge 3$, have the restricted root system of type $\mathrm{B}_2$. Each of them contains a totally geodesic submanifold $SO(5)/(SO(3) \times SO(2))$ of the maximal rank and with the root system of type $\mathrm{B}_2$, so by the arguments similar to the above, the proof of the lemma for these spaces will follow from the proof that $\Phi=0$ for $SO(5)/(SO(3) \times SO(2))$.

Summarising, we see that it suffices to prove the lemma for the Grassmannian $SO(5)/(SO(3) \times SO(2))$ and for the exceptional space $G_2/SO(4)$. This is done below by a direct calculation. 

\medskip 

Let $M_0=SO(5)/(SO(3) \times SO(2))$. Then $\m$ can be identified with the space $M_{3,2}(\br)$ of $3 \times 2$ real matrices with the triple bracket defined by $\ad_{[X,Y]}Z=YX^tZ-XY^tZ+ZX^tY-ZY^tX$. The matrices $E_{a\a}, \; a=1,2,3,\, \a=1,2$ having $1$ in the $a$-th row of the $\a$-th column and zero elsewhere form an orthonormal basis for $\m$ (up to scaling). This basis is acted upon by the isometries from the product of the symmetric groups $S_3 \times S_2 \subset SO(3) \times SO(2)$.

Denote $F(X,Y,Z)$ the operator on the left-hand side of \eqref{eq:biadphi}.

Then from $\<F(E_{11}, E_{21}, E_{32})E_{31}, E_{22}\>=0$ we obtain $\<K_{E_{32}}E_{31},E_{12}\>=0$. Acting by $S_3 \times S_2$ we get $\<K_{E_{a\a}}E_{a\b},E_{b\a}\>=0$, for $\a \ne \b, \, a \ne b$. Next, from $\<F(E_{11}, E_{21}, E_{12})E_{21}, E_{31}\>=0$ we obtain $\<K_{E_{21}}E_{21},E_{32}\>=-\<\Phi(E_{11},E_{12}),E_{31}\>$. Acting by $S_3 \times S_2$ we get $\<K_{E_{a\a}}E_{a\a},E_{b\b}\> =-\<\Phi(E_{c\a},E_{c\b}),E_{b\a}\>$ for $\a \ne \b, \, a \ne b, \; c \ne a,b$. Then from
$\<F(E_{11}, E_{21}, E_{12})E_{32}, E_{21}\>-\<F(E_{21}, E_{31}, E_{22})E_{12}, E_{11}\>=0$ we get that $\<\Phi(E_{11},E_{12}),E_{32}\>=0$, hence $\Phi(E_{a1},E_{a2}) \in \Span(E_{a1},E_{a2})$, for all $a=1,2,3$. Substituting all the above identities to $\<F(E_{12},E_{32},E_{11})E_{12}, E_{11}\> -\<F(E_{22},E_{32},E_{11})E_{22}, E_{21}\>=0$ and to $\<F(E_{12},E_{32},E_{21})E_{12}, E_{21}\> + \<F(E_{11},E_{21},E_{32})E_{11}, E_{21}\>=0$ we obtain that $\<\Phi(E_{32},E_{11}),E_{11}\>+\<\Phi(E_{32},E_{12}),E_{12}\>=0$ and that $\<\Phi(E_{32},E_{11}),E_{11}\>-\<\Phi(E_{32},E_{12}),E_{12}\>=0$ respectively, which implies $\<\Phi(E_{b\b},E_{a\a}),E_{a\a}\>=0$, for all $\a, \b$ and for all $a \ne b$. But then by Lemma~\ref{l:trace}\eqref{it:ltrace1}, we also have that $\<\Phi(E_{a\b},E_{a\b}),E_{a\b}\>=0$, so $\<\Phi(X,E_{a\a}),E_{a\a}\>=0$, for all $\a, a$ and for all $X \in \m$. Moreover, as $\Phi(E_{a1},E_{a2}) \in \Span(E_{a1},E_{a2})$ from the above, we obtain $\Phi(E_{a1},E_{a2})=0$. Furthermore, from $\<F(E_{11},E_{21},E_{32})E_{11}, E_{31}\> - \<F(E_{12},E_{32},E_{21})E_{12}, E_{31}\>=0$ we obtain $\<\Phi(E_{21},E_{32}),E_{31}\>=0$. As $\Phi(E_{a1},E_{a2})=0$, \eqref{eq:cyclePhi} gives that also $\<\Phi(E_{21},E_{31}),E_{32}\>=0$. Acting by $S_3 \times S_2$ we get $\<\Phi(X,E_{a1}),E_{a2}\>=0$, for all $a=1,2,3$ and all $X \in \m$.

From the above we have $\<\Phi(X,E_{a\a}),E_{a\a}\>=0$, for all $\a, a$ and for all $X \in \m$. As the choice of the basis $E_{a\a}$ was arbitrary, this equation still holds with the vector $E_{a\a}$ replaced by any element from its $SO(3) \times SO(2)$ orbit, which implies $\<\Phi(X,Y),Y\>=0$, for all $X, Y \in \m$ such that $Y$ is represented by a $3 \times 2$ matrix of rank one. In particular, it follows that $\<\Phi(X,E_{a\a}),E_{b\a}\>+ \<\Phi(X,E_{b\a}),E_{a\a}\>=0$, for all $\a, a, b$, which by skew-symmetry and \eqref{eq:cyclePhi} gives $\<\Phi(E_{c\a},E_{b\a}),E_{a\a}\>=0$, for all $\a, a, b, c$.

From $\<F(E_{11}, E_{21}, E_{31})E_{11}, E_{22}\>=\<F(E_{31}, E_{21}, E_{22})E_{31}, E_{21}\>=0$ we obtain $\<K_{E_{31}}E_{11}, E_{12}\>= \<K_{E_{31}}E_{21}, E_{22}\>=\<K_{E_{31}}E_{31}, E_{32}\>$, which implies $\<K_XE_{11}, E_{12}\>= \<K_XE_{21}, E_{22}\>= \<K_XE_{31}, E_{32}\>$, for all $X \in \m$, so that all the diagonal elements of the $3 \times 3$ matrix $\<K_XE_{a1}, E_{b2}\>, \; a,b=1,2,3$, are equal. For this property still to hold under the action of the group $SO(3) \subset SO(3) \times SO(2)$, that matrix must be a linear combination of the identity matrix and a skew-symmetric matrix, so in particular, $\<K_XE_{a1}, E_{b2}\> + \<K_XE_{b1}, E_{a2}\>=0$, for all $a \ne b$ and all $X \in \m$. But then from $\<F(E_{11}, E_{21}, E_{32})E_{11}, E_{12}\>-\<F(E_{11}, E_{21}, E_{32})E_{31}, E_{32}\>=0$ we obtain $\<\Phi(E_{12},E_{21}), E_{32}\>=0$. Combining this with the above and acting by $S_3 \times S_2$ we get $\<\Phi(E_{a\a},E_{b\b}), E_{c\a}\>=0$, for all $a,b,c,\a,\b$. The fact that $\<\Phi(E_{a\a},E_{c\a}), E_{b\b}\>=0$, for $\a \ne \b$ then follows from \eqref{eq:cyclePhi}. 

\medskip 

Let $M_0=G_2/SO(4)$. Then $\m$ can be viewed as a Lie triple subsystem of $\mathfrak{so}(7)$ in the following way \cite{Miy}. For $1 \le i \ne j \le 7$, define the matrix $G_{ij} \in \mathfrak{so}(7)$ to have $1$ as its $(i,j)$-th entry, $-1$ as its $(j,i)$-th entry and zero elsewhere. For $i=1, \dots, 7$, define the subspaces $\g_i \subset \mathfrak{so}(7)$ by $\g_i=\{\eta_1 G_{i+1,i+3}+\eta_2 G_{i+2,i+6}+\eta_3 G_{i+4,i+5} \, | \, \eta_1 +\eta_2+\eta_3 =0\}$ (where we subtract $7$ from the subscripts which are greater than $7$). Then $\m=\g_1 \oplus\g_2\oplus\g_5\oplus\g_7$. Every subspace $\g_i$ is abelian; taking $\ag=\g_1$ as a Cartan subspace, we get the restricted root decomposition, with the root vectors $T_1=G_{26}+G_{45}-2G_{13}, \; T_3=G_{35}+G_{67}+2G_{14}, \;T_5=G_{47}-G_{23}-2G_{16}, \; T_4=G_{26}-G_{45}, \; T_6=G_{35}-G_{67}, \; T_2=G_{47}+G_{23}$ (we change the sign of $T_6$ compared with \cite[Eq. (11)]{Miy}). The restricted root system is of type $\mathrm{G}_2$, with $T_1,T_3,T_5$ corresponding to short roots and $T_2,T_4,T_6$, to long roots; the Lie brackets of the vectors $T_i$ are explicitly given in \cite[Table 2]{Miy}. Define $T_7=G_{24}+G_{37}-2G_{56}, \; T_8=G_{24}-G_{37} \in \ag$. With the inner product $\<X,Y\>=\Tr(XY^t)$ (which is proportional to the one induced from the Killing form) the vectors $T_i$ are orthogonal; define $e_i=T_i/\|T_i\|, \; i=1, \dots, 8$. The root vector system has a three-cyclic symmetry defined, for $a=0,1,2$, by $s_{a}e_7= \cos(2a\pi/3)e_7+\sin(2a\pi/3)e_8, \; s_{a}e_8= -\sin(2a\pi/3)e_7+\cos(2a\pi/3)e_8$, and $s_{a}T_i=T_{i+2a}, \; i=1, \dots, 6$ (where we subtract $6$ from the subscripts which are greater than $6$).

Note that the subspace $\m'=\ag\oplus\Span(T_2,T_4,T_6)$ (spanned by $\ag$ and the three long root vectors) is a Lie triple system tangent to a totally geodesic submanifold $SU(3)/SO(3) \subset G_2/SO(4)$ \cite{Kl}.

Denote $F_{ijklm}$ the equation obtained by substituting $X=e_i, \; Y=e_j,\; Z=e_k$ in \eqref{eq:biadphi}, then acting on $e_l$ and taking the inner product of the resulting vector with $e_m$. We abbreviate $K_{e_i}$ to $K_i$ and $\Phi(e_i,e_j)$ to $\Phi_{ij}$ and define an $\m$-valued quadratic form $\theta$ by $\<\theta(X),Y\>=\<\Phi(Y,X),X\>$, for $X,Y \in \m$. Note that $\<\theta(X),X\>=0$.

From $8F_{27846}-3F_{27828}-3F_{25725}+3F_{57858}$ we obtain $\<\Phi_{57},e_5\>=0$, so by $F_{57858}$ we get $\<K_8e_7, e_8\>-3\<\Phi_{78},e_8\> =0$. By the cyclic symmetry, this also holds with the vectors $e_7, e_8$ replaced by $s_{a}e_7, s_{a}e_8$, respectively, so that $\<K_Xe_7, e_8\>- 3\<\Phi_{78},X\> = 0$, for $X=s_{a}e_8= -\sin(2a\pi/3)e_7+\cos(2a\pi/3)e_8, \; a=0,1,2$. It follows that $\<K_Xe_7, e_8\>- 3\<\Phi_{78},X\> = 0$, for all $X \in \ag$. Then from $F_{27827}$ it follows that $\<\Phi_{28},e_2\>=2\<\Phi_{78},e_7\>$. Furthermore, from $2F_{47847}-2\sqrt{3}F_{47826} +2F_{67867}-2\sqrt{3}F_{67824}-F_{27827}$ we obtain $2\<\Phi_{48},e_4\> + 3\<\Phi_{78},e_7\> + 2\<\Phi_{68},e_6\>-\<\Phi_{28},e_2\> = 0$. On the other hand, considering the restriction of equation \eqref{eq:biadphi} to $\m'= \ag\oplus\Span(T_2,T_4,T_6)$ (see Remark~\ref{r:tg}) and applying Lemma~\ref{l:trace}\eqref{it:ltrace1}, with $\m'$ as $\m$, we get $\<\Phi_{48},e_4\>+\<\Phi_{78},e_7\>+\<\Phi_{68},e_6\>+\<\Phi_{28},e_2\> = 0$. It follows that $\<\Phi_{78},e_7\>-3\<\Phi_{28},e_2\> = 0$ which, combined with the equation $\<\Phi_{28},e_2\>=2\<\Phi_{78},e_7\>$ from the above gives $\<\Phi_{28},e_2\>=\<\Phi_{78},e_7\>=0$. By the cyclic symmetry, the second equation implies $\<\Phi_{78},e_8\>=0$. Then, as $\<\Phi_{57},e_5\>=0$ from the above, equation $-3F_{25725}+F_{27828}+3F_{57858}$ gives $\<\Phi_{27},e_2\>=0$. It follows that $\<\theta(e_2),X\> =0$, for all $X \in \ag$, hence by the cyclic symmetry, $\<\theta(e_i),X\> =0$, for all $X \in \ag$ and all long root vectors $e_i$. Moreover, from $F_{25825}+F_{27827}+3F_{57857}$ and $\<\Phi_{28},e_2\>=\<\Phi_{78},e_7\>=0$ we get $\<\Phi_{58},e_5\>=0$. It follows that $\<\theta(e_5),X\> =0$, for all $X \in \ag$, hence again by the cyclic symmetry, $\<\Phi(e_i),X\> =0$, for all $X \in \ag$ and all short root vectors $e_i$.

Summarising the above we get that $\<\theta(Y),X\> =0$, for all $X \in \ag$ and for every $Y$ which is either a root vector, or belongs to $\ag$; in particular,
\begin{equation}\label{eq:g2so4comm}
\<\theta(Y),X\> =0, \quad \text{for all commuting } X,Y \in \m.
\end{equation}
Now, it is easy to see that $e_7$ is a root vector for the Cartan subalgebra $\Span(e_5,e_8)$, so $\<\theta(e_7),e_5\> =0$. Moreover, as $[e_5,e_8]=0$, we have $\<\theta(e_8),e_5\>=0$, by \eqref{eq:g2so4comm}. It follows that $\<\theta(e_7)+\theta(e_8),e_5\>=0$. As the expression on the right-hand side does not depend on the choice of an orthonormal basis for $\ag$, we obtain by cyclic symmetry that $\<\theta(e_7)+\theta(e_8),e_i\>=0$, for every short root vector $e_i$ (that is, for $i=1,3,5$). Similarly, as $e_8$ is a root vector for the Cartan subalgebra $\Span(e_2,e_7)$ and as $[e_2,e_7]=0$, we get $\<\theta(e_7)+\theta(e_8),e_2\>=0$, so by cyclic symmetry, $\<\theta(e_7)+\theta(e_8),e_i\>=0$, for $i=2,4,6$. As $e_7, e_8$ commute, we have $\<\theta(e_8),e_7\>=\<\theta(e_7),e_8\> =0$ by \eqref{eq:g2so4comm}, so $\theta(e_7)+\theta(e_8)=0$. It follows that
\begin{equation}\label{eq:g2so4commt}
\theta(X)+\theta(Y)=0,  \quad \text{for all commuting orthonormal vectors } X,Y \in \m.
\end{equation}
From $\theta(e_7)+\theta(e_8)=0$ it now follows that $\sum_{a=0}^2 \theta(s_a(e_7))=\sum_{a=0}^2 \theta(s_a(e_8))=0$, where, as above, $s_{a}e_7= \cos(2a\pi/3)e_7+\sin(2a\pi/3)e_8, \; s_{a}e_8= -\sin(2a\pi/3)e_7+\cos(2a\pi/3)e_8$, for $a=0,1,2$. By \eqref{eq:g2so4commt} we have $\theta(e_7)+ \theta(e_2)=0$, hence $\theta(s_a(e_7))+\theta(s_a(e_2))=0$, by cyclic symmetry. As $s_1(e_2)=e_4$ and $s_2(e_2)=e_6$ we obtain
\begin{equation}\label{eq:g2so4long}
\theta(e_2)+\theta(e_4)+\theta(e_6)=0.
\end{equation}
Now equation $\sqrt{3}F_{12616}-2F_{25626}$ gives $-\sqrt{3}\<\Phi_{12},e_1\>+\sqrt{3}\<\Phi_{26},e_6\>+2\<\Phi_{25},e_2\>-2\<\Phi_{56}, e_6\> = 0$.
But $\<\Phi_{25},e_2\>=0$ by \eqref{eq:g2so4comm}, as $[e_2,e_5]=0$, and $\<\Phi_{12},e_1\>=-\<\Phi_{42},e_4\>=\<\Phi_{62},e_6\>$ (by \eqref{eq:g2so4comm} and by \eqref{eq:g2so4long}). It follows that $\<\theta(e_6),\sqrt{3}e_2-e_5\> = 0$. On the other hand, equation $\sqrt{3}F_{23434}+2F_{24524}$ gives $-\sqrt{3}\<\Phi_{23},e_3\>-\sqrt{3}\<\Phi_{24},e_4\>+2\<\Phi_{25},e_2\>+2\<\Phi_{45}, e_4\> = 0$.
Again $\<\Phi_{25},e_2\>=0$, and $\<\Phi_{23},e_3\>=-\<\Phi_{26},e_6\>=\<\Phi_{24},e_4\>$ (by \eqref{eq:g2so4comm} and by \eqref{eq:g2so4long}). It follows that $\<\theta(e_4),\sqrt{3}e_2+e_5\> = 0$, which then implies $\<\theta(e_6),\sqrt{3}e_2+e_5\> = 0$ (by \eqref{eq:g2so4long} and \eqref{eq:g2so4comm}). Thus $\<\theta(e_6),e_2\> = \<\theta(e_6),e_5\> = 0$. From the first equation and \eqref{eq:g2so4long} we get $\<\theta(e_6),e_4\> = 0$, so by cyclic symmetry, $\<\theta(e_i),e_j\> = 0$ for all $i,j=2,4,6$. Similarly, the second equation implies $\<\theta(e_i),e_5\> = 0$, for all $i=2,4,6$ by \eqref{eq:g2so4long} and by \eqref{eq:g2so4comm}. Then by cyclic symmetry  $\<\theta(e_i),e_j\> = 0$ for all $i=2,4,6, \; j=1,3,5$, hence $\theta(e_i) \in \ag$, for all $i=2,4,6$. As $[e_2, e_7]=0$, we get from \eqref{eq:g2so4commt} that $\theta(e_7) \in \ag$, which implies $\theta(e_7)=0$ by\eqref{eq:g2so4comm}, so $\theta(s_a(e_7)) = 0$ for $a=0,1,2$ by cyclic symmetry. But then, the restriction of the quadratic form $\theta$ to the two-dimensional space $\ag$ vanishes on three lines in $\ag$, hence $\theta(X)=0$, for all $X \in \ag$.

As any $X \in \m$ belongs to a Cartan subspace it follows that $\theta=0$, that is, $\<\Phi(X,Y),Y\>=0$, for all $X,Y \in \m$. Then the trilinear form $(X,Y,Z) \mapsto \<\Phi(X,Y),Z\>$ is skew-symmetric by the first two arguments and by the second two arguments, hence it is skew-symmetric by all three, which implies $\Phi=0$ by \eqref{eq:cyclePhi}.
\end{proof}

\section{Symmetric spaces of rank one}
\label{s:rk1}

In this section we prove Proposition~\ref{p:irre} for the complex and the quaternionic projective spaces, and also the fact that $\Phi=0$ for the Cayley projective plane (the fact that $K=0$ for $M_0=\Oc P^2$ then follows from Lemma~\ref{l:trace}\eqref{it:ltrace3}). Note that the proof of a statement equivalent to Proposition~\ref{p:irre} for rank one compact symmetric spaces is contained ``in disguise" in \cite{Nampa,Nco} under more general assumptions; for the complex projective space, see \cite{BG1}. For completeness, we give a direct proof here.

\medskip
\underline{$M_0=\bc P^m$.} Denote $J$ the complex structure. Note that $\ad(\h)$ is the centraliser of $J$ in $\so(\m)$, so equation \eqref{eq:Kperpadh} is equivalent to $K_ZJ+JK_Z=0$, for all $Z \in \m$. We have (up to a constant factor) $\ad_{[X,Y]}= X\wedge Y + 2 \<JX,Y\>J+(JX)\wedge(JY)$. Substituting this into \eqref{eq:biadphi} we obtain
\begin{equation}\label{eq:CPm}
    \sigma_{XYZ}(2\<JX,Y\>T_Z+2\<T_ZX,Y\>J+(T_XY-T_YX) \wedge (JZ)+ \Phi(X,Y)\wedge Z)=0,
\end{equation}
where the skew-symmetric operators $T_Z$ are defined by $T_Z=[J,K_Z]$. Note that $T_ZJ+JT_Z=0$, and moreover, that $K=0$ if and only if $T=0$ (as $K_Z=-\frac12 JT_Z$ by \eqref{eq:Kperpadh}).

Consider two cases.

\smallskip
\underline{$m \geqslant 3$.} We first reduce the proof to the case $m=3$. Indeed, let $m > 3$. For a generic triple of vectors $X,Y,Z \in \m$, the subspace $\m' = \Span(X,Y,Z,JX,JY,JZ) \subset \m$ is a Lie triple system tangent to a totally geodesic $\bc P^3 \subset \bc P^m$. Moreover,
if $K$ satisfies condition \eqref{eq:Kperpadh} (so that $K_ZJ+JK_Z=0$), then $K'$ (in the notation of Remark~\ref{r:tg}) also satisfies condition
\eqref{eq:Kperpadh} on $\m'$, as $\m'$ is $J$-invariant. Then, assuming the claim of Proposition~\ref{p:irre} to be true for $M_0=\bc P^3$, we obtain that $\<\Phi(X,Y),Z\> =\<K_XY,Z\>=0$ by Remark~\ref{r:tg}. So we can assume that $m=3$.

For a nonzero $V \in \m$, take $X,Y,Z \perp V,JV$ in \eqref{eq:CPm}, act by the left-hand side on $V$ and take the inner product with $JV$. We obtain $\sigma_{XYZ}\<T_ZX,Y\>=0$ for such $X,Y,Z$. In particular, taking $Y=JX$ and an arbitrary $Z \in \m$ we get $T_XJX=T_{JX}X$. Polarising this equation we obtain
\begin{equation}\label{eq:CPmTXJX}
    T_XY-T_YX=T_{JY}JX-T_{JX}JY, \qquad T_XJX=T_{JX}X,
\end{equation}
for all $X, Y \in \m$.

Taking $Z=e_i$ in \eqref{eq:CPm}, acting by the left-hand side on $e_i$ and summing up by $i$, where $\{e_i\}$ is an orthonormal basis for $\m$ we get (using Lemma~\ref{l:trace}\eqref{it:ltrace1}, the fact that $JT_X+T_XJ=0$ and that $\sum_i T_{e_i}e_i=0$, which follows from \eqref{eq:CPmTXJX}) $3\Phi(X,Y)=T_{JX}Y-T_{JY}X+2\sum_i\<T_{e_i}X,Y\>Je_i$, so
\begin{equation}\label{eq:CPmPhiT}
    3\<\Phi(X,Y),Z\>=\<T_{JX}Y,Z\>-\<T_{JY}X,Z\>-2\<T_{JZ}X,Y\>.
\end{equation}
It follows from (\ref{eq:CPmTXJX}, \ref{eq:CPmPhiT}) that $\Phi(X,JX)=0$ and that $\sigma_{XYZ}\<\Phi(X,Y),JZ\>=0$. Taking the inner product of \eqref{eq:CPm} with $J$ we get $\sigma_{XYZ}\<T_ZX,Y\>=0$, for all $X,Y,Z \in \m$. Then from (\ref{eq:CPmTXJX}, \ref{eq:CPmPhiT}) we obtain $\Phi(X,Y)=J(T_YX-T_XY)$. But then the sum of the last two terms on the left-hand side of \eqref{eq:CPm} commutes with $J$, while the first term anticommutes with $J$ (and the second term vanishes, as $\sigma_{XYZ}\<T_ZX,Y\>=0$). It follows that $\sigma_{XYZ}(\<JX,Y\>T_Z)=0$, which implies $T=0$. It follows that $\Phi=0$ and $K=0$, as required.

\smallskip

\underline{$m=2$.} Then $M_0=\bc P^2$ and we can additionally assume that $\Phi=0$. Taking the inner product of \eqref{eq:CPm} with $J$ and using the fact that $T_ZJ+JT_Z=0$ we obtain that $\sigma_{XYZ}(\<T_ZX,Y\>)=0$. The subspace of those $T \in \so(4)$ which satisfy $TJ+JT=0$ is spanned by two elements $J_2, J_3$ which can be chosen to satisfy $J_2^2=J_3^2=-\id, \; JJ_2=J_3$ (so that $\Span(J,J_2,J_3)$ is one of the factors of $\so(4)=\so(3) \oplus \so(3)$). It follows that $T_Z=\<a,Z\>J_2+\<b,Z\>J_3$ for some $a,b \in \m$. Then the equation $\sigma_{XYZ}(\<T_ZX,Y\>)=0$ implies $\<a,Z\>J_2Z+\<b,Z\>J_3Z+(J_2Z)\wedge a + (J_3Z) \wedge b = 0$. Taking the inner product with $J$ we obtain that $b=-Ja$, so $T_Z=\<a,Z\>J_2Z-\<Ja,Z\>J_3Z$. From \eqref{eq:CPm} we get $\sigma_{XYZ}(2\<JX,Y\>T_Z+(T_XY-T_YX) \wedge (JZ))=0$. Take $Z=a, \; X \perp a, Ja$ and $Y=JX$. Then $T_X=T_Y=0$ and $T_Z=\|a\|^2J_2$ and we obtain $\|a\|^2(2\|X\|^2J_2+(J_3X)\wedge(JX)-(J_2X)\wedge X)=0$. Acting on $X$ we get $3\|a\|^2\|X\|^2J_2X=0$, so $a=0$. It follows that $T_Z=0$ and hence $K=0$, as required.

\medskip
\underline{$M_0=\mathbb{H}P^d, \; d \geqslant 2$.}
Let $(J_1, J_2, J_3=J_1J_2)$ be the quaternionic structure. Define the orthogonal projections $\pi_d: \so(\m) \to \mathfrak{sp}(d)$ and $\pi_1: \so(\m) \to \mathfrak{sp}(1)=\Span(J_1,J_2,J_3)$ by $\pi_d L = \frac14(L-\sum_{i=1}^3 J_iLJ_i)$ and $\pi_1 L = \frac{1}{n}\sum_{i=1}^3 \<J_i,L\>J_i$, where $n=4d$. Clearly $\pi_d\pi_1=\pi_1\pi_d=0$. For $X,Y \in \m$, we have (up to a constant factor)
\begin{equation*}
\ad_{[X,Y]}= X\wedge Y + \sum\nolimits_{i=1}^3(2 \<J_iX,Y\>J_i+(J_iX)\wedge(J_iY))=(n\pi_1+4\pi_d)(X \wedge Y).
\end{equation*}
Substituting this into \eqref{eq:biadphi} we obtain
\begin{equation}\label{eq:HPm} 
    \sigma_{XYZ}([n\pi_1+4\pi_d,\ad_{K_Z}] (X \wedge Y)+\Phi(X,Y)\wedge Z)=0.
\end{equation}
By condition \eqref{eq:Kperpadh}, for all $X \in \m$, $K_X$ belongs to the $(\mathfrak{sp}(1)\oplus\mathfrak{sp}(d))$-module $(\mathfrak{sp}(1)\oplus\mathfrak{sp}(d))^\perp \subset \so(\m)$, which gives $\pi_1K_X=\pi_dK_X = 0$, that is, $K_X\perp J_i$ and $K_X=\sum_{i=1}^3 J_iK_XJ_i$. Moreover, $\pi_1\ad_{K_X}\pi_1=\pi_d\ad_{K_X}\pi_1=\pi_1\ad_{K_X}\pi_d=\pi_d\ad_{K_X}\pi_d=0$.

Therefore projecting \eqref{eq:HPm} to $\mathfrak{sp}(d)$ and to $\mathfrak{sp}(1)$ we obtain $\pi_d\sigma_{XYZ}(4\ad_{K_Z} (X \wedge Y)+ \Phi(X,Y)\wedge Z)=0$ and $\sigma_{XYZ}\pi_1(n\ad_{K_Z} (X \wedge Y)+\Phi(X,Y)\wedge Z)=0$, respectively, which gives
\begin{gather}\label{eq:hpmidmA}
    \pi_d \sigma_{XYZ} (T(X,Y)\wedge Z)=0, \quad \text{where } T(X,Y)=4(K_XY-K_YX)+\Phi(X,Y),\\
    \sigma_{XYZ} \<J_iZ,n(K_XY-K_YX)+\Phi(X,Y)\>=0. \label{eq:hpmpi}
\end{gather}

The above equations still hold in $\m^\bc$, the complexification of $\m$, if we extend all the maps and the inner product by complex linearity. 
We have $\m^\bc=E_{\i} \oplus E_{-\i}$, where $E_{\pm \i}$ are the ($\pm \i$)-eigenspaces of $J_1$. The subspaces $E_{\pm \i}$ are of dimension $2d$ and are isotropic relative to the inner product. Moreover, the operators $J_2, J_3$ interchange the subspaces $E_{\pm \i}$, and for any $X \in E_{\ve \i}, \; \ve = \pm 1$, we have $J_3X=-\ve \i J_2X$. Substituting $X_j \in E_{\ve_j \i}, \; j=1,2,3, \; \ve_j = \pm 1$, as $X, Y, Z$ into \eqref{eq:hpmidmA} we obtain
\begin{equation}\label{eq:hpmidmAc}
    \sigma_{123} (((\id+ \ve_3 \i J_1)T(X_1,X_2))\wedge X_3+(J_2(\id+ \ve_3 \i J_1)T(X_1,X_2))\wedge (J_2X_3))=0.
\end{equation}
Note that $(\id+ \ve \i J_1)Y$ is twice the $E_{-\ve \i}$-component of $Y \in \m^\bc$. First consider the case when $\ve_1=\ve_2=\ve_3=\ve$. Acting by the left-hand side of \eqref{eq:hpmidmAc} on $Y \in E_{\ve \i}$ such that $\<Y,J_2X_j\>=0, j =1,2,3$ (such a nonzero $Y$ exists, as $\dim E_{\ve \i}=2d \ge 4$), we obtain that $\<T(X_1,X_2),Y\>=0$, hence the $E_{-\ve \i}$-component of $T(X_1,X_2)$ lies in $\Span(J_2X_1,J_2X_2,J_2X_3)$, for any linearly independent $X_1,X_2,X_3 \in E_{\ve \i}$, therefore it lies in $J_2 \Span(X_1,X_2)$. As $\dim E_{\ve \i}=2d \ge 4$, it follows that the $E_{-\ve \i}$-component of $T(X_1,X_2)$ equals $J_2 (X_1 \wedge X_2) p_{-\ve}$, for all $X_1,X_2 \in E_{\ve \i}$, for some $p_{-\ve} \in E_{-\ve \i}$ by Lemma~\ref{l:cross}\eqref{it:cr1}. Now suppose that $\ve_1=\ve_2=-\ve_3=\ve$ in \eqref{eq:hpmidmAc}. Acting by the left-hand side of \eqref{eq:hpmidmAc} on $Y \in E_{\ve \i}$ such that$\<Y,J_2X_1\>=\<Y,J_2X_2\>=\<Y,X_3\>=0$, we obtain $\<T(X_1,X_2),J_2Y\>=\<T(X_1,X_3),Y\>=0$, for any $X_1,X_2 \in E_{\ve \i}, \; X_3 \in E_{-\ve \i}$ such that $X_1,X_2,J_2X_3$ are linearly independent. From the first equation it follows that the $E_{\ve \i}$-component of $T(X_1,X_2)$ lies in $\Span(X_1,X_2)$, hence it equals $(X_1 \wedge X_2) q_{-\ve}$, for all $X_1,X_2 \in E_{\ve \i}$, where $q_{-\ve} \in E_{-\ve \i}$. From the second equation it follows that the $E_{-\ve \i}$-component of $T(X_1,X_3)$ lies in $\Span(J_2X_1,X_3)$, so it equals $\<X_3,a_\ve\>J_2X_1+\<J_2X_1,b_\ve\> X_3$, for all $X_1 \in E_{\ve \i}, \; X_3 \in E_{-\ve \i}$, where $a_\ve, b_\ve \in E_{-\ve \i}$. Combining these we find that there exist $p_j \in \m^\bc, \; j=0,1,2,3$, such that $T(X,Y)=(X \wedge Y)p_0 + \sum_{j=1}^3 J_j(X \wedge Y)p_j$.

As $T(X,Y)$ is real when $X$ and $Y$ are real, we obtain that $p_j \in \m$. Substituting into \eqref{eq:hpmidmA}, taking the inner product of the resulting equation with $J_1$ and choosing $X,Y \perp p_0,p_2,p_3$ we get $\<(\<Z,p_0\>J_1+\<Z,p_2\>J_3-\<Z,p_3\>J_2)X, Y\>=0$. But the operator in the brackets is either zero or nonsingular, and in the latter case its maximal isotropic subspace has dimension $n/2 <n-3$. It follows that $\<Z,p_0\>J_1+\<Z,p_2\>J_3-\<Z,p_3\>J_2=0$, so $p_0=p_2=p_3=0$. Similar argument with $J_1$ replaced by $J_2$ shows that also $p_1=0$. Hence $T=0$, so $4(K_XY-K_YX)=-\Phi(X,Y)$, which by \eqref{eq:cyclePhi} implies
\begin{equation}\label{eq:hpmkphi}
    4\<K_XY,Z\>=\<\Phi(Y,Z),X\>.
\end{equation}
Let $\m' \subset \m$ be a four-dimensional $\mathfrak{sp}(1)$-invariant subspace. We have $J_i\m'\subset \m'$, and so $\m'$ is a Lie triple system tangent to a totally geodesic sphere. Restricting $K$ and $\Phi$ to $\m'$ as in Remark~\ref{r:tg} we obtain from \eqref{eq:biadphi} that the projection of $\sigma_{XYZ}(\Phi(X,Y)\wedge Z)$ to $\so(\m')$ is zero, for all $X,Y,Z \in \m'$. By Lemma~\ref{l:cross}\eqref{it:cr2}, this implies that $\<\Phi(X,Y), Z\>=0$, for all $X,Y,Z \in \m'$. It now follows from \eqref{eq:hpmkphi} that $\<K_XY,Z\>=0$, for all $X,Y,Z \in \m'$. Now suppose $\m_1', \m_2' \subset \m$ are four-dimensional, orthogonal, $\mathfrak{sp}(1)$-invariant subspaces. Then $\ad_{[U,V]}\m_1' \subset \m_1'$, for $U, V \in \m_2'$, so for $X,Y,Z \in \m_1'$ we get $\<(n\pi_1+4\pi_d)(\ad_{K_Z}) (X \wedge Y), U \wedge V\>=\<(K_Z X) \wedge Y-(K_Z Y) \wedge X, (n\pi_1+4\pi_d)(U \wedge V)\>=\<(K_Z X) \wedge Y-(K_Z Y) \wedge X, \ad_{[U,V]}\>=-2\<\ad_{[U,V]}Y, K_Z X \> +2\<\ad_{[U,V]}X, K_Z Y\>=0$, as $\ad_{[U,V]}\m_1' \subset \m_1'$ and $K_Z \m_1' \perp \m_1'$ since $Z \in \m_1'$. Then taking the inner product of equation~\eqref{eq:HPm} with $U \wedge V$, where $X,Y,Z \in \m_1', \; U, V \in \m_2'$, and using \eqref{eq:hpmkphi} we obtain $\sigma_{XYZ} \sum_{i=1}^3 \<J_iX,Y\>\<\Phi(V, J_iU) - \Phi(U, J_iV), Z\>=0$. Now taking $X=J_1p, \, Y=J_2p, \, Z=J_3p$, for some $p \in \m_1'$ we get $\sum_{i=1}^3 \<\Phi(V, J_iU) - \Phi(U, J_iV), J_ip\>=0$. This is true for any $p \perp \m_2'$, but also for $p \in \m_2'$, as $\<\Phi(\m_2',\m_2'),\m_2'\>=0$. It follows that $\sum_{i=1}^3 J_i(\Phi(V, J_iU) - \Phi(U, J_iV))=0$, for all $U,V \in \m_2'$, so taking $U=J_1V$ we get $J_2(-\Phi(V,J_3V)+\Phi(J_2V,J_1V))+J_3(\Phi(V,J_2V)+\Phi(J_3V,J_1V))=0$. On the other hand, from \eqref{eq:hpmkphi} and the fact that for all $Z \in \m$, $\pi_d K_Z = 0$ we obtain that $\Phi(X,Y)+\sum_{i=1}^3 \Phi(J_iX,J_iY)=0$, for all $X, Y \in \m$; substituting $Y=J_2X$ and $Y=J_3X$ we get $\Phi(X,J_2X)+\Phi(J_1X,J_3X)=\Phi(X,J_3X)+\Phi(J_2X,J_1X)=0$. Then the above equation gives $-J_2\Phi(V,J_3V)+J_3\Phi(V,J_2V)=0$ which implies $J_2\Phi(V,J_2V)+J_3\Phi(V,J_3V)=0$, for all $V \in \m$. As a similar equation is satisfied for any two of three subscripts $1,2,3$ we obtain that $\Phi(V,J_iV)=0$, for all $V \in \m$ and all $i=1,2,3$.

Now from \eqref{eq:hpmpi}, \eqref{eq:cyclePhi} and \eqref{eq:hpmkphi} we obtain $\sigma_{XYZ} \<J_iZ,\Phi(X,Y)\>=0$. Substituting $Y=J_iX$ and using the fact that $\Phi(X,J_iX)=0$ we obtain $\<X,\Phi(X,Z)\>=-\<J_iX,\Phi(J_iX,Z)\>$, for all $X, Z \in \m$ and all $i=1,2,3$. But then $\<X,\Phi(X,Z)\>=-\<J_1X,\Phi(J_1X,Z)\>=\<J_2J_1X,\Phi(J_2J_1X,Z)\>=\<J_3X,\Phi(J_3X,Z)\>=-\<X,\Phi(X,Z)\>$. It follows that $\<\Phi(X,Z),X\>=0$, for all $X, Z \in \m$, so the map $(X,Y,Z) \mapsto \<\Phi(X,Y),Z\>$ is skew-symmetric and hence $\Phi=0$ by \eqref{eq:cyclePhi}. Then by \eqref{eq:hpmkphi}, $K=0$.

\medskip
\underline{$M_0=\Oc P^2$.} For any nonzero vector $X \in \m$, the stabiliser of $X$ in $H=\mathrm{Spin}(9)$ is $\mathrm{Spin}(7)$; it acts transitively on the unit spheres in the root spaces (that is, in the eigenspaces of the Jacobi operator $(R_0)_X$) \cite[Corollary~2.26a]{Nag}. It follows that for any nonzero root vector $Y$, both $X$ and $Y$ belong to a Lie triple system $\m' \in \m$ tangent to a totally geodesic $\mathbb{H}P^2  \subset \Oc P^2$. By Remark~\ref{r:tg} and from the above proof we get $\<\Phi(X,Y),X\>=0$, which by linearity implies $\<\Phi(X,Z),X\>=0$, for any $Z \in \m$. It follows that the map $(X,Y,Z) \mapsto \<\Phi(X,Y),Z\>$ is skew-symmetric and so $\Phi=0$ by \eqref{eq:cyclePhi}.

\section{Spaces $SU(3)/SO(3)$ and $SL(3)/SO(3)$}
\label{s:su3so3}

In the case when $M_0=SU(3)/SO(3)$ or $SL(3)/SO(3)$, the claim of Proposition~\ref{p:main} (and Proposition~\ref{p:irre}) is false, as we show below. Note however, that if a symmetric space $M_0$ is reducible and contains $SU(3)/SO(3)$ or $SL(3)/SO(3)$ as one of the factors (and has no factors of constant curvature), then the claims of both the Theorem and Proposition~\ref{p:main} still hold.

Let $M_0=SL(3)/SO(3)$ (the dual case is similar). Then $\m$ is the space of symmetric traceless $3 \times 3$ matrices and $\h=\so(3)$. We have the following $\h$-module irreducible orthogonal decomposition: $\Lambda^2 \m = P_7 \oplus \ad(\h)$. By \eqref{eq:Kperpadh} $K_X \in P_7$, for all $X \in \m$. It follows that $K \in \m \otimes P_7$. The latter is a $35$-dimensional $\so(3)$-module whose irreducible decomposition is well-known \cite{Fr2, BN}: $\m \otimes P_7 \simeq P_3 \oplus P_5 \oplus P_7 \oplus P_9 \oplus P_{11}$, where $P_{2l+1}$ is the unique irreducible $\so(3)$-module of dimension $2l+1$. The solution space $P$ of the pairs $(K, \Phi)$ which satisfy equations (\ref{eq:cyclePhi}, \ref{eq:biadphi}, \ref{eq:Kperpadh}) is an $\so(3)$-module. As by equation~\eqref{eq:biadphi}, $\Phi$ is uniquely determined by $K$ and satisfies \eqref{eq:cyclePhi} (Remark~\ref{rem:KPhi}), the module $P$ is isomorphic to a certain submodule of $\m \otimes P_7$. A computer assisted calculation shows that $\dim P = 14$. The module $P$ contains a three-dimensional submodule defined as follows: for $L \in \so(3)$, set
\begin{equation*}
\<K_XY,Z\>=\Tr((\tfrac15 X[Z,Y]-ZXY)L), \qquad \<\Phi(Y,Z),X\>=\Tr((X[Z,Y]+2ZXY)L).
\end{equation*}
(the fact that (\ref{eq:cyclePhi}, \ref{eq:biadphi}, \ref{eq:Kperpadh}) are satisfied can be checked directly). From the above decomposition it follows that the complementary submodule is an irreducible $\so(3)$-module isomorphic to $P_{11}$. Taking the real part of the highest weight vector in the complexification of $\m \otimes P_7$ we obtain, relative to the orthonormal basis $e_1=(E_{12}+E_{21})/\sqrt{2}, \, e_2=(E_{13}+E_{31})/\sqrt{2}, \, e_3=(E_{23}+E_{32})/\sqrt{2}, \,e_4=(E_{11}-E_{22})/\sqrt{2}, \, e_5=(E_{11}+E_{22}-2E_{33})/\sqrt{6}$ for $\m$, that this module is defined by the element
\begin{equation*}
    K_{e_1}=\left(
              \begin{array}{ccccc}
                0 & 0 & -1 & 0 & 0 \\
                0 & 0 & 0 & -1 & 0 \\
                1 & 0 & 0 & 0 & 0 \\
                0 & 1 & 0 & 0 & 0 \\
                0 & 0 & 0 & 0 & 0 \\
              \end{array}
            \right), \;
    K_{e_4}=\left(
              \begin{array}{ccccc}
                0 & 1 & 0 & 0 & 0 \\
                -1 & 0 & 0 & 0 & 0 \\
                0 & 0 & 0 & -1 & 0 \\
                0 & 0 & 1 & 0 & 0 \\
                0 & 0 & 0 & 0 & 0 \\
              \end{array}
            \right), \; K_{e_2}=K_{e_3}=K_{e_5}=0, \\
\end{equation*}
and then $\<\Phi(Y,Z),X\>=\frac32 \<K_XY,Z\>$, for all $X,Y,Z \in \m$ (and again, the fact that (\ref{eq:cyclePhi}, \ref{eq:biadphi}, \ref{eq:Kperpadh}) are satisfied can be checked directly).

Based on the fact that the dimension of the solution space is large, one may suggest that there indeed exists a five-dimensional Riemannian space having the same Weyl tensor as the symmetric space $SL(3)/SO(3)$ (or $SU(3)/SO(3)$), but not conformally equivalent to it. Note that such a space, if it exists, must not be Einstein (by Lemma~\ref{l:trace}\eqref{it:ltrace3}) and must have a constant scalar curvature (by Lemma~\ref{l:trace}\eqref{it:ltrace1} and \eqref{eq:gdefPhi}).


\begin{thebibliography}{BKV}

\bibitem[BG1]{BG1}
Bla\v zi\'c N., Gilkey P.
\emph{Conformally Osserman manifolds and conformally complex space forms},
Int. J. Geom. Methods Mod. Phys. \textbf{1} (2004), 97 -- 106.

\bibitem[BG2]{BG2}
Bla\v zi\'c N., Gilkey P.
\emph{Conformally Osserman manifolds and self-duality in Riemannian geometry}.
Differential geometry and its applications, 15 -- 18, Matfyzpress, Prague, 2005.

\bibitem[BKV]{BKV}
Boeckx E., Kowalski O., Vanhecke L.
\emph{Nonhomogeneous relatives of symmetric spaces},
Differential Geom. Appl. \textbf{4} (1994), 45 -- 69.

\bibitem[BN]{BN}
Bobie\'{n}ski M., Nurowski P.
\emph{Irreducible $\mathrm{SO(3)}$ geometry in dimension five},
J. Reine Angew. Math. \textbf{605} (2007), 51 -- 93.

\bibitem[Der]{Der}
Derdzinski A.
\emph{Exemples de m\'{e}triques de K\"{a}hler et d'Einstein auto-duales sur le plan complexe},
G\'{e}om\'{e}trie riemannienne en dimension $4$ (S\'{e}minaire Arthur Besse 1978/79).
Cedic/Fernand Nathan, Paris (1981), 334 -- 346.

\bibitem[Fr1]{Fr1}
Friedrich T.
\emph{Weak $\mathrm{Spin}(9)$-structures on $16$-dimensional Riemannian manifolds},
Asian J. Math. \textbf{5} (2001), 129 -– 160.

\bibitem[Fr2]{Fr2}
Friedrich T.
\emph{On types of non-integrable geometries},
Rend. Circ. Mat. Palermo (II) Suppl. \textbf{71} (2003), 99 -- 113.

\bibitem[Gil]{Gil}
Gilkey P. \emph{The geometry of curvature homogeneous pseudo-Riemannian manifolds}.
ICP Advanced Texts in Mathematics, 2. Imperial College Press, London, 2007.

\bibitem[Hel]{Hel}
Helgason S.
\emph{Differential geometry, Lie groups, and symmetric spaces}.
Pure and Applied Mathematics, 80. Academic Press, Inc. New York -- London, 1978.

\bibitem[Kl]{Kl}
Klein S.
\emph{Totally geodesic submanifolds in Riemannian symmetric spaces},
Differential geometry, 136 -- 145, World Sci. Publ., Hackensack, NJ, 2009.

\bibitem[KTV]{KTV}
Kowalski O., Tricerri F., Vanhecke L.
\emph{Curvature homogeneous Riemannian manifolds},
J. Math. Pures Appl. \textbf{71} (1992), 471 -– 501.

\bibitem[Kur]{Kur}
Kurita M.
\emph{On the holonomy group of the conformally flat Riemannian manifold},
Nagoya Math. J. \textbf{9} (1955), 161 -- 171.

\bibitem[Lis]{Lis}
Lister W. G.
\emph{A structure theory of Lie triple systems},
Trans. Amer. Math. Soc. \textbf{72} (1952), 217--242.

\bibitem[Miy]{Miy}
Miyaoka R.
\emph{Geometry of $G_2$ orbits and isoparametric hypersurfaces}, Nagoya Math. J., \textbf{203} (2011), 175 -- 189.

\bibitem[Nag]{Nag}
Nagano T.
\emph{The involutions of compact symmetric spaces. II},
Tokyo J. Math. \textbf{15} (1992), 39--82.

\bibitem[NT]{NT}
Nagano T., Tanaka M.
\emph{The involutions of compact symmetric spaces. V},
Tokyo J. Math. \textbf{23} (2000), 403 -- 416. 

\bibitem[N1]{Ndga}
Nikolayevsky Y.
\emph{Weyl homogeneous manifolds modelled on compact Lie groups},
Diff. Geom. Appl. \textbf{28} (2010), 689 -- 696.

\bibitem[N2]{Nco}
Nikolayevsky Y.
\emph{Conformally Osserman manifolds}, Pacific J. Math. \textbf{245}(2010), 315 -- 358.

\bibitem[N3]{Nampa}
Nikolayevsky Y.
\emph{Conformally Osserman manifolds of dimension $16$ and Weyl-Schouten Theorem for rank-one symmetric spaces}, Ann. Mat. Pura Appl. (2012) (to appear).

\bibitem[Pan]{Pan}
Panyushev D.
\emph{Isotropy representations, eigenvalues of a Casimir element, and commutative Lie subalgebras},
J. London Math. Soc. (2) \textbf{64} (2001), 61 -- 80.

\bibitem[PY]{PY}
Panyushev D., Yakimova O.
\emph{Symmetric pairs and associated commuting varieties},
Math. Proc. Cambridge Philos. Soc. \textbf{143} (2007), 307 -- 321.

\bibitem[Rot]{Rot}
Roter W.
\emph{On conformally symmetric spaces with positive definite metric forms},
Bull. Acad. Polon. Sci. Sér. Sci. Math. Astronom. Phys. \textbf{24} (1976), 981 -- 985.

\bibitem[Slu]{Slu}
Slupinski M. J.
\emph{A Hodge type decomposition for spinor valued forms},
Ann. Sci. \'{E}cole Norm. Sup. (4) \textbf{29} (1996), 23 -- 48.

\bibitem[Sza]{Sza}
Szab\'{o}, Z. I.
\emph{Structure theorems on Riemannian spaces satisfying $R(X,\,Y)\cdot R=0$. I. The local version},
J. Differential Geom. \textbf{17} (1982), 531 -- 582.

\bibitem[Tam]{Tam}
Tamaru H.
\emph{The local orbit types of symmetric spaces under the actions of the isotropy subgroups},
Differential Geom. Appl. \textbf{11} (1999), 29 -- 38.

\bibitem[TV]{TV}
Tricerri F., Vanhecke L. \emph{Curvature homogeneous Riemannian manifolds},
Ann. Sci. École Norm. Sup. (4) \textbf{22} (1989), 535 -- 554.

\bibitem[Yam]{Yam}
Yamaguti K.
\emph{On the cohomology space of Lie triple system}, Kumamoto J. Sci. Ser. A \textbf{5} (1960) 44 -- 52.

\end{thebibliography}
\end{document}